\documentclass[12pt,a4paper,english]{smfart}
\usepackage[english]{babel}
\usepackage{ragged2e}
\usepackage{smfthm, mathabx}
\usepackage{stmaryrd}
\usepackage{amssymb}
\usepackage{amsmath}
\usepackage{amsthm}
\usepackage{amsfonts}
\usepackage{graphicx}
\usepackage{enumerate}
\usepackage{comment}
\usepackage{url}

\usepackage{tikz-cd}

\usepackage{appendix}

\usepackage{euscript,mathrsfs}
\usepackage[utf8]{inputenc} 
\usepackage{longtable}
\usepackage{dsfont}

\usepackage[OT2,T1]{fontenc}

\usepackage[all,cmtip]{xy}
\usepackage{alltt}

\usepackage{calrsfs}

\makeatletter
\def\thm@space@setup{%
  \thm@preskip=\parskip \thm@postskip=0pt
}
\makeatother

\xyoption{all}

\usepackage{float}

\usepackage{fullpage}

\usepackage{color}

\newtheorem{Theorem}{Theorem}

\newtheorem{Exa}{Example}

\newtheorem*{Corollary}{Corollary}

\author[Oussama Hamza]{Oussama Hamza}
\address{Institute for Advanced Studies in Mathematics, Harbin Institute of Technology, Harbin, China 150001}
\email{ohamza3@uwo.ca}

\author{Christian Maire}
 \address{Université Marie et Louis Pasteur,  CNRS, Institut FEMTO-ST, F-25000 Besançon, France}
\email{christian.maire@univ-fcomte.fr}

\author[J{\'a}n Min{\' a}{\v c}]{J{\'a}n Min{\' a}{\v c}}
\address{Department of Mathematics, Western University, London, Ontario, Canada N6A5B7}
\email{minac@uwo.ca}

\author[Nguy{\^ e}n Duy T{\^ a}n]{Nguy{\^ e}n Duy T{\^ a}n}
\address{Faculty of Mathematics and Informatics, Hanoi University of Science and Technology, No 1, Dai Co Viet Road, Hanoi, Vietnam}
\email{tan.nguyenduy@hust.edu.vn}

\title{Maximal $2$-extensions of Pythagorean fields and Right Angled Artin Groups}

\thanks{
The first and second author are grateful to the Institute for Advanced Studies in Mathematics (IASM) at Harbin Institute of Technology for support during the summer 2025. The first, second
and third authors acknowledge the support of the Western Academy for Advanced Research
(WAFAR) during the year 2022/23. The second author was partially supported by the EIPHI Graduate School (contract “ANR-17-EURE-0002") and by the Bourgogne Franche-Comté Region. The third author was partially supported by the Natural Sciences and Engineering Research Council of Canada (NSERC)
grant R0370A01. He gratefully acknowledges the Western University Faculty of Science Distinguished Professorship 2020-2021, and the current support of the Fields Institute for research in Mathematical Sciences. The fourth author is partially
supported by an AMS-Simons Travel Grant, and the Vietnam National Foundation
for Science and Technology Development (NAFOSTED) under grant number 101.04-2023.21.}

\newcommand{\F}{\mathbb{F}}
\newcommand{\Z}{\mathbb{Z}}
\newcommand{\NN}{\mathbb{N}}
\newcommand{\RR}{\mathbb{R}}

\newcommand{\C}{\mathbb{C}}

\def\G{{\rm G}}

\def\Gal{{\rm Gal}}

\def\bz{{\mathbf{z}}}

\def\bX{{\mathbf{X}}}
\def\bE{{\mathbf{E}}}

\def\O{{\mathcal O}}
\def\Rr{{\mathcal R}}
\def\PP{{\mathcal P}}
\def\J{{\mathcal J}}
\def\E{{\mathcal E}}
\def\Ll{{\mathcal L}}

\def\A{{\mathcal A}}
\def\AA{{\mathbb A}}

\def\I{{\mathcal I}}

\def\B{{\mathcal B}}

\def\L{{\rm L}}

\begin{document}

\begin{abstract}
In this paper, we describe minimal presentations of maximal pro-$2$ quotients of absolute Galois groups of formally real Pythagorean fields of finite type. For this purpose, we introduce a new class of pro-$2$ groups: $\Delta$-Right Angled Artin groups. 

We show that maximal pro-$2$ quotients of absolute Galois groups of formally real Pythagorean fields of finite type are $\Delta$-Right Angled Artin groups. Conversely, let us assume that a maximal pro-$2$ quotient of an absolute Galois group is a $\Delta$-Right Angled Artin group. We then show that the underlying field must be Pythagorean, formally real and of finite type. As an application, we provide an example of a pro-$2$ group which is not a maximal pro-$2$ quotient of an absolute Galois group, although it has Koszul cohomology and satisfies both the Kernel Unipotent and the strong Massey Vanishing properties.

 We combine tools from group theory, filtrations and associated Lie algebras, profinite version of the Kurosh Theorem on subgroups of free products of groups, as well as several new techniques developed in this work.
\end{abstract}

\maketitle

A central open problem in Galois theory and number theory is to classify profinite groups that arise as absolute Galois groups. One natural approach is to study properties that such groups must satisfy. On the one hand, Artin–Schreier~\cite{artin1927kennzeichnung} showed that the only non-trivial finite subgroups of absolute Galois groups are of order $2$. On the other hand, the well known Bloch-Kato conjecture imposes strong restrictions on Galois cohomology. This conjecture was proved in $2011$ by Rost and Voevodsky (see~\cite{haesemeyer2019norm} 
and~\cite{056424e1-724a-3d58-b863-a9f884daae3d}). For some other results concerning absolute Galois groups of some special fields and some related topics in field arithmetic, we refer the interested reader to~\cite{efrat1994galois}, \cite{bary2015sylow}, \cite{haran2021absolute}, and \cite{fried2005field}.

Positselski~\cite{positselski2014galois} proposed a strengthened version of the Bloch–Kato conjecture. 
Let us fix a prime~$p$. A profinite group is said to satisfy the Koszul property (at~$p$) if its cohomology algebra over~$\F_p$ (with trivial action) is Koszul. This means that the cohomology algebra is generated in degree~$1$, with relations of degree~$2$, and admits a linear resolution. Positselski conjectured that absolute Galois groups, of fields containing a~$p$-th root of unity, satisfy the Koszul property. 
This property has significant applications in current algebra and geometry. We refer to~\cite{polishchuk2005quadratic} for a comprehensive exposition. 

In~\cite{minavc2015kernel} and~\cite{minavc2016triple}, the last two authors proposed further conjectures describing properties on absolute Galois groups: the Massey Vanishing and the Kernel Unipotent properties. The Massey Vanishing conjecture has attracted considerable attention. For instance, Harpaz and Wittenberg~\cite{harpaz2023massey} proved it for all algebraic number fields. We refer to~\cite{merkurjev2024lectures} for a complete exposition on the Massey Vanishing conjecture and recent results. 

This paper focuses on the class $\PP$ of maximal pro-$2$ quotients of absolute Galois groups of formally real Pythagorean fields of finite type (RPF fields). This class is deeply connected with Witt rings, orderings of fields, and the Milnor conjecture. We refer to the work of the third author and Spira \cite{minac1986galois}, \cite{minac1990formally}, \cite{minachilbseries} and \cite{minac1996witt}, Marshall \cite{marshall1979classification}, Jacob \cite{jacob1981structure}, Efrat-Haran~\cite{efrat1994galois} and Lam~\cite{lam2005introduction}. The third author in \cite{minac1986galois} and \cite{minac1986thesis} described the class $\PP$. It is the minimal class of pro-$2$ groups containing~$\Delta\coloneq \Z/2\Z$, stable under coproducts and some semi-direct products. 
With Spira \cite{minac1990formally} and \cite{minac1996witt}, they also showed that groups in~$\PP$ are characterized by finite quotients. Precisely their third Zassenhaus quotients. These results led to several consequences.  A first application was given by the last two authors with Pasini and Quadrelli. They showed that~$\PP$ satisfies the Koszul property \cite{minac2021koszul}. A second application was given by Quadrelli~\cite[Theorem~$1.2$]{quadrelli2023massey}. He proved that~$\PP$~satisfies the Massey Vanishing property.

Building on these results, this work aims to study presentations of groups in~$\PP$. As an application, we give several new examples of groups which are not pro-$2$ maximal quotients of absolute Galois groups.
Our approach relates $\PP$ to the well-known class of pro-$p$ Right Angled Artin Groups (RAAGs). We refer to~\cite{bartholdi2020right} for a general introduction. This class has recently played an important role in Galois theory.
We refer to the work of Snopce and Zalesski \cite[Theorem~$1.2$]{snopce2022right} and the work of~Blumer, Quadrelli and Weigel~\cite[Theorem~$1.1$]{blumer2023oriented}.

\subsection*{The class of $\Delta$-RAAGs}
Let us denote by~$x_0$ the generator of the multiplicative group~$\Delta\coloneq \Z/2\Z$. Intuitively, a~$\Delta$-Right Angled Artin group ($\Delta$-RAAG) is a semi-direct product of a (pro-$2$) Right Angled Artin Group (RAAG) by $\Delta$. The action inverts a "natural" set of generators up to conjugacy. 

Let~$\Gamma\coloneq (\bX, \bE)$ be an undirected graph with~$d_\Gamma$ vertices~$\{1,\dots, d_\Gamma\}$ and~$r_{\Gamma}$ edges. We recall that a pro-$2$ Right Angled Artin Group $G_\Gamma$ 
is defined by a pro-$2$ presentation with
$d_\Gamma$ generators $\{x_1,\dots, x_{d_\Gamma}\}$ 
and relations~$\{\lbrack x_i,x_j\rbrack \coloneq x_i^{-1}x_j^{-1}x_ix_j\}_{\{i,j\}\in \bE}$. Fix $\bz\coloneq (z_{i})_{i=1}^{d_\Gamma}$, a $d_\Gamma$-uplet in~$G_\Gamma$. We assume that there exists an action~$\delta_\bz\colon \Delta \to {\rm Aut}(G_\Gamma)$, which is well-defined and satisfies the condition:
\begin{equation}\tag{conj}\label{conj}
\delta_\bz(x_0)(x_i)\coloneq (x_i^{-1})^{z_{i}}\coloneq z_i^{-1}x_i^{-1}z_i=\lbrack z_i,x_i\rbrack x_i^{-1}, \quad \text{for } 1\leq i \leq d_\Gamma.
\end{equation}
Set $G_{\bz}\coloneq G_\Gamma \rtimes_{\delta_\bz}\Delta$. We define the class of~$\Delta$-RAAGs as the class of all pro-$2$ groups given by~$G_{\bz}$ where~$\Gamma$ varies along all graphs. 
As a consequence, the presentation of the pro-$2$ group~$G_{\bz}$ is given by $d_{\Gamma}+1$ generators and $r_{\Gamma}+d_{\Gamma}+1$ relations:
\begin{multline}\tag{$\bz$-Pres} \label{z-Pres}
G_\bz\coloneq \langle x_0, x_1,\dots, x_{d_\Gamma} |\quad \lbrack x_u,x_v\rbrack=1, \lbrack x_0,x_i^{-1}\rbrack x_i^2\lbrack x_i,z_i\rbrack=1, x_0^2=1, 
\\ \text{ for } \{u,v\}\in \bE, 1\leq i \leq d_\Gamma\rangle.
\end{multline}
We have explicit examples from \cite[Proposition $3.16$]{hamza2023zassenhaus}: every graph~$\Gamma$ and family~$\bz_0\coloneq (1,\dots, 1)$ gives a well-defined group~$G_{\bz_0}\coloneq G_\Gamma \rtimes_{\delta_{\bz_0}}\Delta$. In that case, the condition~\eqref{conj} is defined by:
$$\delta_{\bz_0}(x_0)(x_i)\coloneq x_i^{-1}, \quad \text{for } 1\leq i \leq d_\Gamma.$$

A major innovation of this paper is the introduction of the action $\delta_\bz$, in the definition of $\Delta$-RAAGs.
Using the action~$\delta_\bz$ and a profinite version of the Kurosh Subgroup Theorem, we show that the class of $\Delta$-RAAGs is stable under coproducts. This is Theorem~\ref{technical coprod}. From Proposition~\ref{rem semitriv2}, we also observe that this class is stable under some specific semi-direct products. As a consequence, we conclude from \cite{minac1986galois} that~$\PP$ is a subclass of $\Delta$-RAAGs. 
The study of $\Delta$-RAAGs is also motivated by Proposition~\ref{lower RAAGs}, which allows us to recover the Zassenhaus filtration of a $\Delta$-RAAG from its underlying graph.

\subsection*{Our results}
Let $K$ be a field of characteristic different from~$2$. Denote by $G_K$ the maximal pro-$2$ quotient of its absolute Galois group. Define $L\coloneq K(\sqrt{-1})$. 
We assume that $K$ is a formally real Pythagorean field of finite type (abbreviated RPF field). This means that $(i)$~$-1$ is not a square, $(ii)$ the sum of two squares is a square, $(iii)$~the group~$K^{\times}/K^{\times 2}$ is finite.

Observe that $\Delta \simeq \Gal(L/K)$ and we have an exact sequence of pro-$2$ groups:
\begin{equation}\label{split Pyt}
1\to G_L \to G_K \to \Delta \to 1.
\end{equation}

The following result is analogous to to \cite[Theorem~$1.2$]{snopce2022right} and \cite[Theorem~$1.1$]{blumer2023oriented}, and also relies on graph theory. It gives a new property satisfied by pro-$2$ maximal quotients of absolute Galois groups.

\begin{Theorem}[Theorems~\ref{mainpyt} and \ref{Graph and Pyt}]\label{Theo fonda}
If $K$ is a RPF field, then the exact sequence~\eqref{split Pyt} splits. The pro-$2$-group $G_L$ is RAAG, and $G_K$ is $\Delta$-RAAG. Conversely, let~$K$ be a field, and assume that~$G_K$ is a~$\Delta$-RAAG. Then~$K$ is a RPF field. Moreover,~$G_K$ is uniquely determined by its underlying graph.
\end{Theorem}

From presentations of pro-$2$ groups given by Theorem~\ref{Theo fonda},
we get new examples of pro-$2$ groups in~$\PP$  (Example~\ref{intermexam}). We also infer new examples of pro-$2$ groups which are not maximal pro-$2$ quotients of absolute Galois groups (Example~\ref{second set of examples}). Furthermore, if~$K$ is a RPF field, Theorem~\ref{Theo fonda} shows that the group~$G_K$ admits a quadratic presentation. This presentation is given by~\eqref{z-Pres}. This allows us to positively answer a question raised by Weigel~\cite{weigel652koszul} for groups in~$\PP$. See Proposition~\ref{koszulpyt}. From~\cite[Proposition $1$]{hamza2023extensions} and~\cite{leoni2024zassenhaus}, this question is related to the Positselski conjecture. Let us highlight the following consequence from the~$\Delta$-RAAG theory.

\begin{Corollary}[Proposition \ref{koszulpyt}]\label{Ans Weigel}
Assume that~$G$ is in $\PP$. Then the presentation~\eqref{z-Pres} of~$G$, given by Theorem~\ref{Theo fonda} is minimal.  Let us denote by~$\Gamma$ the underlying graph of~$G$. Then:
$$H^\bullet(G,t)\coloneq \sum_n \dim_{\F_2}H^n(G) t^n=\frac{\Gamma(t)}{1-t}.$$
Here $\Gamma(t)\coloneq \sum_n c_n(\Gamma)t^n$, with $c_n(\Gamma)$ the number of $n$-cliques of $\Gamma$, i.e. maximal complete subgraphs of~$\Gamma$ with $n$ vertices.
\end{Corollary}
Let us recall that~$d_\Gamma$ and~$r_\Gamma$ are the number of vertices and edges of~$\Gamma$. The previous corollary tells us that the minimal number of generators and relations of $G$ is~$d_\Gamma+1$ and~$r_\Gamma+d_\Gamma+1$.

We conclude this paper with the first known example of a pro-$2$ group which satisfies the Koszul, the (strong) Massey Vanishing and the Kernel Unipotent properties, but is not a maximal pro-$2$ quotient of an absolute Galois group. This example is studied in detail in Section~\ref{final example}.

\begin{Theorem}[Theorem~\ref{counterexample}]\label{impo exam}
The pro-$2$ group
\begin{multline*}
G\coloneq \langle x_0, x_1, x_2,x_3,x_4| \quad \lbrack x_1,x_2\rbrack =\lbrack x_2,x_3\rbrack =\lbrack x_3,x_4\rbrack = \lbrack x_4,x_1\rbrack =1,
\\ x_0^2=1, \quad x_0x_jx_0x_j=1, \forall j\in [\![1;4]\!]\rangle
\end{multline*}
 does not occur as a maximal pro-$2$ quotient of an absolute Galois group. But it satisfies the Koszul, the (strong) Massey Vanishing and the Kernel Unipotent properties.
\end{Theorem} 

\subsection*{Acknowledgments}
The first three authors are thankful to Mathieu Florence and Claudio Quadrelli for several discussions and for their interest in this work. They are also grateful to Michael Rogelstad for several comments. Furthermore, they acknowledge Federico Scavia, Thomas Weigel and Olivier Wittenberg for discussions. The first author is grateful to Elyes Boughattas, Baptiste Cerclé, Keping Huang, Mahmoud Sharkawi and Jun Wang for their interest in this work. Furthermore, he is especially grateful to Priya Bucha Jain, Manimugdha Saikia and Sayantan Roy Chowdhury for their encouragement and support. The second and third authors are grateful to Alexander Merkurjev and Federico Scavia for discussions related to Massey products and
absolute Galois properties.

\tableofcontents

\section*{List of notations}
This table summarizes the main symbols and notations used in the paper.
Most of these notations were also used in~\cite{hamza2023zassenhaus} and~\cite{hamza2023extensions}.

\subsection*{Group and Galois theoretic notations} We start with some algebraic notations. All subgroups of a profinite group will be considered closed. 
\renewcommand{\arraystretch}{1.2}
\begin{longtable}{|p{0.35\textwidth}|p{0.6\textwidth}|}
\hline
\textbf{Symbol / Notation} & \textbf{Description} \\ \hline
$K$, $L$ & $K$ a field of characteristic different from~$2$, $L\coloneq K(\sqrt{-1})$. \\ \hline
$G_K$, $G_L$ & The maximal pro-$2$ quotient of the absolute Galois group of $K$, $L$. \\ \hline
RPF fields & Formally real Pythagorean fields of finite type. \\ \hline
$\PP$ & The class of maximal pro-$2$ quotients of absolute Galois groups of RPF fields. \\ \hline
$\Delta \coloneq \mathbb{Z}/2\mathbb{Z}$ & Multiplicative group with two elements, generated by~$x_0$. \\ \hline
$F(d), F$ & Free pro-2 group on $d$ generators; written as $F$ when $d$ is clear. \\ \hline
$R$ & Normal closed subgroup of $F$ generated by the relations $l_1, \ldots, l_r$. \\ \hline
$G\coloneq F/R$ & A finitely presented pro-$2$ group presented by generators $\{x_0, x_1, \ldots, x_d\}$ and relations $\{l_1, \ldots, l_r\}$. \\ \hline
$x^y \coloneq y^{-1}xy$ & Conjugation of $x$ by $y$. \\ \hline
$[x, y] \coloneq x^{-1}y^{-1}xy$ & Commutator of $x$ and $y$. \\ \hline
$[H, H]$ & Normal closure of the subgroup generated by commutators of elements in~$H$, a subgroup of~$G$. \\ \hline
$H^2$ & Normal closure of the subgroup generated by all squares in~$H$, a subgroup of~$G$. \\ \hline
$H_1H_2$ & Normal closure of the subgroup generated by $ab$, for~$a\in H_1$, $b\in H_2$ and $H_1$ and $H_2$ subgroups of~$G$. \\ \hline
$\mathbb{U}_{n}$ & Group of $n\times n$ upper triagular unipotent matrices over $\mathbb{F}_2$. \\ \hline
$\mathbb{I}_{n}$ & Identity matrix of size $n\times n$ over~$\F_2$. \\ \hline
$H^n(G)$ & The $n$-th continuous cohomology group of the trivial $G$-module~$\mathbb{F}_2$. \\ \hline
$H^\bullet(G)\coloneq \bigoplus_{n\in \NN} H^n(G)$ & Graded cohomology ring. \\ \hline
$H^\bullet(G, t)\coloneq \sum_{n} \dim_{\mathbb{F}_2} H^n(G)t^n$ & Poincaré series of~$G$. It is defined when $\dim_{\mathbb{F}_2} H^n(G)$ is finite for every~$n$ in~$\NN$. \\ \hline
$E(G)$ & Completed group algebra of $G$ over $\mathbb{F}_2$. \\ \hline
$E_n(G)$ & The $n$-th power of the augmentation ideal of $E(G)$, for~$n$ in~$\NN$. \\ \hline
 $G_n\coloneq \{ g \in G\mid g-1 \in E_n(G) \}$ & Zassenhaus filtration of~$G$. \\ \hline
\end{longtable}

\subsection*{Lie algebras associated to filtrations} We continue with some notations on Lie algebras coming from the Zassenhaus filtrations:
\renewcommand{\arraystretch}{1.2}
\begin{longtable}{|p{0.35\textwidth}|p{0.6\textwidth}|}
\hline
\textbf{Symbol / Notation} & \textbf{Description} \\ \hline
$\Ll(G)\coloneq \bigoplus_{n\in \NN} \Ll_n(G)$ & Graded Lie algebra where $\Ll_n(G)\coloneq G_n / G_{n+1}$. \\ \hline
$\E(G)\coloneq \bigoplus_{n\in \NN} \E_n(G)$ & Graded algebra where $\E_n(G)\coloneq E_n(G) / E_{n+1}(G)$. \\ \hline
$gocha(G,t)\coloneq \sum_{n\in \NN}c_n(G)t^n$ & Gocha series of~$G$, where~$c_n\coloneq \dim_{\mathbb{F}_2} E_n(G)/E_{n+1}(G)$. This name comes from Golod and Shafarevich.\\ \hline
$\Ll$, $E$, $\E$ & Restricted free graded Lie algebra, algebra of noncommutative polynomials, and noncommutative series algebra on $\{X_0, \ldots, X_d\}$ over $\mathbb{F}_2$, where each~$X_i$ has degree~$1$. \\ \hline
$\psi\colon E(F_{d+1})\to E$ & Magnus isomorphism sending $x_i \mapsto 1 + X_i$. \\ \hline
$l$, $n_l$ & We denote by $l$ an element in $F$ and $n_l$ is the least integer~$n$ such that $l$ is in $F_n\setminus F_{n+1}$. \\ \hline
$I$, $\I$ and $\J$ & Ideal (closed, graded and Lie-graded) generated by $\{\psi(l)-1\mid l\in R\}$, the image of $\psi(l)-1$ in $E_{n_l}/E_{n_l+1}$ for $l$~in $R$ and the image of $l$ in $F_{n_l}/F_{n_l+1}$.  \\ \hline
$\psi_G \colon E(G)\to E/I$ & It comes from the Magnus isomorphism. Observe that we have: $\E(G)\simeq \E/\I$ and $\Ll(G)\simeq \Ll/\J$. \\ \hline
$\lbrack \bullet, \bullet\rbrack$ & Lie bracket of a Lie algebra. \\ \hline
\end{longtable}

\subsection*{Graph theoretical notations} We finish by notations from graph theory:
\renewcommand{\arraystretch}{1.2}
\begin{longtable}{|p{0.35\textwidth}|p{0.6\textwidth}|}
\hline
\textbf{Symbol / Notation} & \textbf{Description} \\ \hline
$\Gamma \coloneq (X, E)$ & An undirected graph with set of vertices set $X$ and set of edges $E$. \\ \hline
$\Gamma_0$ & The graph with one vertex. \\ \hline
$\Gamma_1\simeq \Gamma_2$ & We say that we have a graph isomorphism between~$\Gamma_1$ and~$\Gamma_2$, if we have a bijection between the vertices of~$\Gamma_1$ and~$\Gamma_2$ which preserves the edges. \\ \hline
$d_\Gamma, r_\Gamma$ & Numbers of vertices and edges of $\Gamma$. \\ \hline
$\Gamma^f, \Gamma^c$ & The graphs $\Gamma^f$ and $\Gamma^c$ are the free and the complete graphs on~$d_\Gamma$ vertices. The graph~$\Gamma^f$ does not have edges and the graph~$\Gamma^c$ has~$\frac{d_\Gamma(d_\Gamma-1)}{2}$ edges.\\ \hline
$\nabla$ & Join of two graphs. \\ \hline
$\coprod$ & Either coproduct of pro-$p$ groups or disjoint union of graphs. \\ \hline
$c_n(\Gamma)$ & Number of $n$-cliques (maximal complete subgraphs with $n$ vertices), for a positive integer~$n$. \\ \hline
$\Gamma(t)\coloneq \sum_{n\in \NN}c_n(\Gamma)t^n$ & Clique polynomial. \\ \hline
$G_\Gamma$ & Pro-2 RAAG (Right Angled Artin Group) associated to the graph $\Gamma$. \\ \hline
$\E_\Gamma$, $\Ll_\Gamma$, $E_\Gamma$ & Quotients of (Lie, noncommutative) algebras $\E$, $\Ll$ and $E$ by the (Lie) ideal generated by $\{\lbrack X_u, X_v\rbrack\}_{\{u,v\}\in \bE}$. We denote the gradation of $\E_\Gamma$ by~$\{\E_{\Gamma,n}\}_{n\in \NN}$ and the gradation of $\Ll_{\Gamma}$ by $\{\Ll_{\Gamma,n}\}_{n\in \NN}$. \\ \hline
$\bz\in G_\Gamma^{d_\Gamma}$, and $\delta_\bz\colon \Delta \to {\rm Aut}(G_\Gamma)$ & A $d_\Gamma$-tuple $\bz \coloneq (z_1, \ldots, z_{d_\Gamma})$, and an action $\delta_{\bz}$ satisfying the equality~\eqref{conj}.
 \\ \hline
$G_{\bz} \coloneq G_\Gamma \rtimes_{\delta_\bz} \Delta$ & $\Delta$-RAAG defined by an underlying graph $\Gamma$ and the action $\delta_\bz$. These groups are presented by~\eqref{z-Pres}.\\ \hline
$\chi_0, \psi_1, \dots, \psi_{d_\Gamma}$ & $\chi_0$ and~$\psi_i$ are the characters in~$H^1(G_\bz)$ associated to~$x_0$ and~$x_i$, for~$1\leq i \leq d_\Gamma$.\\ \hline
$Gr(\PP)$ & The class of underlying graphs coming from~$\PP$. This is justified by~Theorem~\ref{Theo fonda}. Precisely the graph~$\Gamma$ lies in~$Gr(\PP)$ if and only if there exists a RPF field~$K$ such that $G_\Gamma=G_L$, with $L\coloneq K(\sqrt{-1})$. \\ \hline
\end{longtable}

\section{Introductory results on $\Delta$-RAAGs}
In this section, we study the Zassenhaus filtration on $\Delta$-RAAGs.

\subsection{Filtrations on $\Delta$-RAAGs}
Let $G_{\bz}$ be a~$\Delta$-RAAG. We start this section by studying the subgroup~$G_{\bz,2}=G_{\bz}^2\lbrack G_{\bz}, G_{\bz}\rbrack$ which is also known as the Frattini subgroup of~$G_{\bz}$.

\begin{lemm}\label{lemm low}
We have~$\lbrack G_{\bz}, G_{\bz}\rbrack =G_{\bz,2}=G_{\Gamma,2}=G_\Gamma^2$.
\end{lemm}

\begin{proof}
Let us first notice that~$G_{\Gamma}$ is an open subgroup of~$G_{\bz}$. Thus~$G_{\Gamma,2}$ is an open subgroup of~$G_{\bz,2}$.
From now, we need the following equalities:
$$ (i)~\lbrack x,y\rbrack =x^{-2}(xy^{-1})^2y^2, \quad \text{and } (ii)~\lbrack x_0,x_i^n\rbrack =x_0x_i^{-n}x_0x_i^{n}\equiv x_i^{2n} \pmod{\lbrack G_\Gamma,G_{\Gamma}\rbrack},$$
where $n$ is an integer. 
The identity $(i)$ is well-known and shows that~$\lbrack G,G\rbrack \subset G^2$ for every group~$G$. From~$(i)$ and~$(ii)$, we observe that~$\lbrack G_{\bz}, G_{\bz}\rbrack= G_{\Gamma,2}=G_\Gamma^2$, and~$G_{\bz,2}=G_{\bz}^2$.

To conclude, let us show that~$G_{\bz, 2}\subset G_{\Gamma,2}$. Take $x\in G_{\bz,2}$ and~$N$ an open subgroup of~$G_\bz$. There exists elements~$u_1,\dots, u_k$, positive integers~$n_1,\dots, n_k$ and elements~$x_{j_l}$ in $\{x_0,\dots, x_d, x_1^{-1},\dots, x_d^{-1}\}$, such that:
$$xN = u_1^2\dots u_k^2N, \quad \text{for }u_i= x_{j_1}\dots x_{j_{n_i}}.$$
Consequently, $xN= \prod_i x_{j_1}^2\dots x_{j_{n_i}}^2 uN$ for some  $u\in \lbrack G_{\bz}, G_{\bz}\rbrack\subset G_{\Gamma,2}$. Let us observe that~$\prod_i x_{j_1}^2\dots x_{j_{n_i}}^2$ is in~$G_{\Gamma,2}$. Then for every open subgroup~$N$ of~$G_{\bz}$, we have~$x\in G_{\Gamma,2} N$.
 Since~$G_{\Gamma,2}$ is open in~$G_\bz$, we conclude that $x$ is an element in $G_{\Gamma,2}$.
\end{proof}

As a consequence of the three-subgroup lemma (\cite[§$0.3$]{DDMS}), we have the following result:

\begin{lemm}\label{three subgroup}
Let~$n$ be a positive integer. We have the following equality:
$$\lbrack G_\bz,G_{\Gamma,n}\rbrack G_{\Gamma,n+1}= \lbrack G_{\Gamma},G_{\Gamma,n}\rbrack G_{\Gamma,n+1}=G_{\Gamma,n+1}.$$
\end{lemm}

\begin{proof}
Let~$G$ be a pro-$2$ group. We introduce $\{\gamma_n(G)\}_{n\in \NN}$ the lower central series of~$G$. It is defined by $\gamma_1(G)\coloneq G$ and~$\gamma_n(G)\coloneq \lbrack G, \gamma_{n-1}(G)\rbrack$, for every positive integer~$n\geq 2$.

From an inductive argument and the three-subgroup lemma (\cite[§$0.3$]{DDMS}), we observe that for every positive integer $n$, we have $\lbrack x_0,\gamma_n(G_\Gamma)\rbrack \leq G_{\Gamma,n+1}$.  Thus using Lazard's formula \cite[Theorem~$11.2$]{DDMS}, we infer $\lbrack x_0, G_{\Gamma,n}\rbrack \subset G_{\Gamma,n+1}$. Consequently 
$$\lbrack G_\bz ,G_{\Gamma,n}\rbrack G_{\Gamma,n+1}= \lbrack G_{\Gamma},G_{\Gamma,n}\rbrack G_{\Gamma,n+1}=G_{\Gamma,n+1}.$$
\end{proof}

Let us now compute the Zassenhaus filtrations of $\Delta$-RAAGs.
\begin{prop}\label{lower RAAGs}
For every integer~$n\geq 2$, we have: $G_{\Gamma,n}=G_{\bz,n}.$

As a consequence, we infer that for every positive integer~$n\geq 2$:
$$\Ll_n(G_\bz)\coloneq G_{\bz,n}/G_{\bz,n+1}\simeq \Ll_{\Gamma,n}.$$
\end{prop}

\begin{proof}
The proof is done by induction. Lemma~\ref{lemm low} gives us~$G_{\bz,2}=G_{\Gamma,2}$.

Assume $n\geq 2$ and let us state our induction hypothesis: $G_{\Gamma,n}=G_{\bz,n}$. Then from \cite[Theorem~$12.9$]{DDMS}, the induction hypothesis, and Lemma~\ref{three subgroup}, we infer:
\begin{equation*}
\begin{aligned}
G_{\bz,n+1} G_{\Gamma,n+1}&= G_{\bz, \lceil \frac{n+1}{2} \rceil}^2 \prod_{i+j=n+1}\lbrack G_{\bz,i},G_{\bz,j}\rbrack G_{\Gamma,n+1}
\\  &= G_{\Gamma, \lceil \frac{n+1}{2} \rceil}^2 \prod_{i+j=n+1}\lbrack G_{\Gamma,i}, G_{\Gamma,j}\rbrack G_{\Gamma,n+1}
\\   &=G_{\Gamma,n+1}.
\end{aligned}
\end{equation*}
Therefore $G_{\bz,n+1}\leq G_{\Gamma,n+1}$. 
Furthermore, the group~$G_\Gamma$ is a subgroup of~$G_{\bz}$, so for every positive integer~$n$, we have $G_{\Gamma,n}\leq G_{\bz,n}$. Thus, for every integer~$n\geq 2$,
$G_{\bz,n}=G_{\Gamma,n}$.

From \cite[Theorem~$2.4$]{Tra}, we infer~$\Ll(G_\Gamma)\simeq \Ll_\Gamma$. This allows us to conclude. 
\end{proof}

Let us recall the following result:

\begin{lemm}\label{Graph identification}
Let~$\Gamma_1$ and~$\Gamma_2$ be two graphs. The following assertions are equivalent:

$(i)$ $G_{\Gamma_1}\simeq G_{\Gamma_2}$.

$(ii)$ $\E_{\Gamma_1}\simeq \E_{\Gamma_2}$.

$(iii)$ $H^\bullet(\Gamma_1)\simeq H^\bullet(\Gamma_2)$.

$(iv)$ $\Gamma_1\simeq\Gamma_2$.
\end{lemm}

\begin{proof}
Clearly, $(iv)$ implies $(i)$.

The first author showed in \cite[Proposition~$3.4$] {hamza2023extensions} that~$\E(G_\Gamma)\simeq \E_\Gamma$. Thus $(i)$ implies~$(ii)$.

Since~$\E_\Gamma\simeq \E(G_\Gamma)$, the assertion $(ii)$ gives us~$\E(G_{\Gamma_1})\simeq \E(G_{\Gamma_2})$. These two algebras are Koszul. We refer to~\cite[Theorem~$1.2$]{bartholdi2020right}. Therefore from~\cite[Proposition~$3.4$]{hamza2023extensions}, we infer~$H^\bullet(G_{\Gamma_1})\simeq H^\bullet(G_{\Gamma_2})$. So $(ii)$ implies $(iii)$.

We conclude $(iii)$ implies $(iv)$ using \cite[Corollary $2.14$]{Tra}. We also refer to~\cite{kim}.
\end{proof}

As a consequence:

\begin{coro}\label{graph iso}
Let us consider two graphs $\Gamma_1$ and $\Gamma_2$, and two vectors $\bz_1\in \G_{\Gamma_1}^{d_{\Gamma_1}}$ and~$\bz_2\in G_{\Gamma_2}^{d_{\Gamma_2}}$. If we have an isomorphism of $\Delta$-RAAGs, $G_{\bz_1}\simeq G_{\bz_2}$, then $\Gamma_1\simeq \Gamma_2$.
\end{coro}

\begin{proof}
Since $G_{\bz_1}\simeq G_{\bz_2}$, we observe that $d_{\Gamma_1}=d_{\Gamma_2}$. In particular, we infer that:
$$G_{\Gamma_1}/G_{\Gamma_1,2}\simeq G_{\Gamma_2}/G_{\Gamma_2,2}.$$
Moreover, from Proposition \ref{lower RAAGs}, we infer for every integer $n\geq 2$, that $G_{\Gamma_1,n}=G_{\Gamma_2,n}$. Thus $\E_{\Gamma_1}\simeq \E(G_{\Gamma_1})\simeq \E(G_{\Gamma_2})\simeq \E_{\Gamma_2}$. We conclude using~Lemma~\ref{Graph identification}.
\end{proof}

\subsection{Third Zassenhaus quotient of $\Delta$-RAAGs}
In this section, we study the group~$\mathcal{G}_{\bz}\coloneq G_{\bz}/G_{\bz,3}$. When~$G$ is a pro-$2$ group, we denote by~$G^4$ the normal closure of the subgroup of~$G$ generated by the~$4$-th powers. Observe that we have the equality~$G_3=  G^4\lbrack G^2, G\rbrack$. Let us recall that the third Zassenhaus quotient~$G/G_3$ plays a fundamental role in the study of the maximal pro-$2$ quotients of absolute Galois groups of RPF fields. In that context, it is called the Witt group of~$G$. We refer to \cite{minac1990formally} and \cite{minac1996witt} for further details. 

\begin{lemm}
Let $G_\bz$ be a $\Delta$-RAAG. For every integer $n\geq 1$, we have $\delta_\bz(x_0)(G_{\Gamma,n})\subset G_{\Gamma,n}$. Consequently $\delta_\bz$ induces an action $\overline{\delta_\bz}$ of $\Delta$ on $G_\Gamma/G_{\Gamma,3}$.
\end{lemm}

\begin{proof}
Let us consider two elements $x$ and $z$ in $G_\Gamma$. Recall that \cite[Proposition~$1.7$]{hamza2023extensions} and the Magnus isomorphism gives us an isomorphism~$\psi_{G_\Gamma}\colon E(G_\Gamma)\to E_{\Gamma}$. 
Let us write~$\psi_{G_\Gamma}(x)\coloneq 1+X$ and $\psi_{G_\Gamma}(z)\coloneq 1+Z$ in $E_\Gamma$. 
We infer:
$$\psi_{G_\Gamma}(x^z)= 1+X+\lbrack X,Z\rbrack \times \frac{1}{1+Z}.$$
Thus for every positive integer $n$, we deduce $\delta_\bz(x_0)(G_{\Gamma,n})\subset G_{\Gamma,n}$, and so $\delta_\bz$ induces an action $\overline{\delta_\bz}$ on $G_\Gamma/G_{\Gamma,3}$.
\end{proof}

\begin{coro}
 We have $\mathcal{G}_{\bz}\simeq G_\Gamma/G_{\Gamma,3}\rtimes_{\overline{\delta_\bz}} \Delta.$
\end{coro}

\begin{proof}

    From Proposition \ref{lower RAAGs}, we infer:
$$G_\Gamma\cap G_{\bz,3}=G_\Gamma\cap G_{\Gamma,3}=G_{\Gamma,3}.$$
    
    Consequently, we obtain an exact sequence:
    $$1\to G_\Gamma/G_{\Gamma,3} \to \mathcal{G}_{\bz}\to \Delta \to 1.$$
    Furthermore,~$x_0$ is not in~$G_{\bz,3}$. Thus the previous exact sequence splits and the action of~$\Delta$ on $G_\Gamma/G_{\Gamma,3}$ is given by $\overline{\delta_\bz}$.
\end{proof}

\begin{coro}\label{Delta and Pyt}
    We have $\mathcal{G}_{\bz}^2=\lbrack \mathcal{G}_{\bz}, \mathcal{G}_{\bz}\rbrack.$
\end{coro}

\begin{proof}
    Since $\mathcal{G}_\bz$ is a $2$-group, we have the inclusion $\lbrack \mathcal{G}_{\bz}, \mathcal{G}_{\bz}\rbrack \subset \mathcal{G}_\bz^2$. We also have:
        \begin{flalign*}
            \lbrack \mathcal{G}_{\bz}, \mathcal{G}_{\bz}\rbrack &= \lbrack G_{\bz}/G_{\bz,3}, G_{\bz}/G_{\bz,3}\rbrack
            \\ &=\lbrack {G}_{\bz}, {G}_{\bz}\rbrack G_{\bz,3}/G_{\bz,3}
            \\ &=G_{\bz}^2G_{\bz,3}/G_{\bz,3}
            \\ &=\mathcal{G}_{\bz}^2.
        \end{flalign*}
\end{proof}

\begin{rema}
Similar results to the previous parts were already known for some groups in~$\PP$. We refer to \cite[Corollary $4.8$ and Lemma $4.12$]{Minac}.
\end{rema}

\subsection{Gocha series and associated graded algebras}
Now, we aim to compute the algebras~$\E(G_{\bz})$ and~$\Ll(G_{\bz})$ as a quotient of~$\E$ and~$\Ll$. The main tools are the Magnus isomorphism and restricted Lie algebras techniques. We refer to~\cite[Chapitre~II, Partie~$3$]{LAZ} and~\cite[§12.1]{DDMS} for further detail. We first observe that~$\E(G)$ is the universal envelope of~$\Ll(G)$. From the Magnus isomorphism, we also have the following chain of morphisms of~$\F_2$-vector spaces~$G_\Gamma /G_{\Gamma,2}\hookrightarrow F(d_\Gamma +1)/F(d_{\Gamma}+1)_2\simeq E/E_2\hookrightarrow \Ll\hookrightarrow \E$.
 
We define $\I_{\bz}$ (resp.~$\J_{\bz}$) the two-sided ideal of $\E$ (resp.\ $\Ll$) generated by:
$$\{X_0^2, \quad \lbrack X_u,X_v],\quad \text{and} \quad [X_0,X_k]+[X_k,\epsilon_k]+X_k^2\}_{\{u,v\}\in \bE \text{ and } 1\leq k \leq d},$$
where $\epsilon_k$ is the image of $z_k$ in~$G_\Gamma$ through the previous chain of maps.
Let us observe that~$\epsilon_k=0$ if and only if $z_k$ has degree~$n_{z_k}$ larger than or equal two. Moreover, since~$z_k$ is in~$G_\Gamma$, then $\epsilon_k$ can be seen as an element in $\Ll_{\Gamma}$. We introduce 
$$\E_{\bz}\coloneq \E/\I_{\bz}, \quad \text{and} \quad \Ll_{\bz}\coloneq \Ll/\J_{\bz}.$$

\begin{theo}\label{gocha Delta Raag}
For every $\Delta$-RAAG $G_\bz$, we have the following isomorphisms of graded-$\F_2$ Lie algebras:
$$\E(G_\bz)\simeq \E_\bz, \quad \text{and} \quad \Ll(G_{\bz})\simeq \Ll_{\bz}.$$
\end{theo}

\begin{proof}
As vector spaces, we have from Proposition \ref{lower RAAGs}:
$$\Ll_1(G_{\bz})\simeq \F_2 \bigoplus \Ll_1(G_{\Gamma}), \quad \text{and for } n\geq 2, \quad \Ll_n(G_{\bz})\simeq \Ll_n(G_{\Gamma}).$$

From the relations satisfied by restricted Lie algebras (see for instance \cite[Equation~$(7)$ in~§$12.1$]{DDMS}) and relations~$X_0^2=0$ and $\lbrack X_0, X_k\rbrack +\lbrack X_k,\epsilon_k\rbrack+X_k^2=0$; we observe for every~$n\geq 2$ that~$\Ll_{\bz,n}\subset \Ll_{\Gamma,n}$. Thus $\dim_{\F_2}\Ll_{\bz,n}\leq \dim_{\F_2}\Ll_{\Gamma,n}$.

Furthermore we have graded surjections~$\E_{\bz}\to \E(G_\bz)$ and~$\Ll_{\bz}\to \Ll(G_\bz)$. Thus~$\dim_{\F_2}\Ll_n(G_\bz)\leq \dim_{\F_2} \Ll_{\bz,n}$.
Consequently, from Proposition~\ref{lower RAAGs} and the previous discussion, we deduce for every $n\geq 2$ that:
$$\dim_{\F_2}\Ll_n(G_\bz)\leq \dim_{\F_2} \Ll_{\bz,n} \leq \dim_{\F_2} \Ll_{\Gamma,n} = \dim_{\F_2}\Ll_n(G_\bz).$$

This implies that $\Ll_\bz\simeq \Ll(G_\bz)$. Since $\E_{\bz}$ is the universal envelope of $\Ll_\bz$ and $\E(G_\bz)$ is the universal envelope of~$\Ll(G_\bz)$, we infer $\E_\bz\simeq \E(G_\bz)$.
\end{proof}

\begin{rema}\label{quad defined}
Following the terminology of \cite[Definition~$7.7$]{minac2021koszul}, Theorem~\ref{gocha Delta Raag} shows that the presentation~\eqref{z-Pres} is quadratically defined.
\end{rema}

\begin{coro}\label{gocha Delta Raag2}
We have 
$$gocha(G_{\bz},t)=\frac{1+t}{\Gamma(-t)}.$$
\end{coro}

\begin{proof}
Proposition~\ref{lower RAAGs} gives us $\Ll_1(G_{\bz})\simeq \F_2 \bigoplus \Ll_{\Gamma,1}$, and for  n$\geq 2$, $\Ll_n(G_{\bz})\simeq \Ll_{\Gamma,n}$.
Then, from Jennings-Lazard formula \cite[Proposition $3.10$, Appendice A]{LAZ}, we obtain:
$$gocha(G_{\bz},t)=(1+t)\times gocha(G_{\Gamma},t)=\frac{1+t}{\Gamma(-t)}.$$
\end{proof}

\subsection{Example: $\delta_{\bz_0}$-action}
We take $\bz_0\coloneq (1,\dots, 1)$ in~$G_\Gamma^{d_\Gamma}$. From \eqref{conj}, the action~$\delta_{\bz_0} \colon \Delta \to {\rm Aut}(G_\Gamma)$ is defined by~$\delta_{\bz_0}(x_0)(x_i)\coloneq x_i^{-1}$. By~\cite[Proposition~$3.16$]{hamza2023zassenhaus}, this action is well-defined.  Let~$\Gamma^f$ be a free graph on~$d_\Gamma$ vertices and~$\Gamma^c$ be the complete graph on~$d_\Gamma$ vertices. We observe the following facts.

$\bullet$ We have an isomorphism:
$$\coprod_{i=1}^{d_\Gamma+1} \Delta \simeq F(d_\Gamma) \rtimes_{\delta_{\bz_0}} \Delta\coloneq G_{\Gamma^f}\rtimes_{\delta_{\bz_0}}\Delta.$$

$\bullet$ Every $\bz$ in $G_{\Gamma^f}^{d_\Gamma}$ defines a $\Delta$-RAAG $G_\bz$.

$\bullet$ Every $\bz$ in $G_{\Gamma^c}^{d_\Gamma}$ defines a $\Delta$-RAAG $G_\bz$. Since~$G_{\Gamma^c}$ is commutative, we infer:
$$G_\bz \simeq \Z_2^{d_\Gamma}\rtimes_{\delta_{\bz_0}}\Delta\simeq G_{\Gamma^c}\rtimes_{\delta_{\bz_0}}\Delta.$$

\begin{defi}[SAP and superpythagorean groups]
Let us define $\bz_0\coloneq (1,\dots,1)\in G_\Gamma^{d_\Gamma}$. We say that a $\Delta$-RAAG $G_\bz$ is:
\begin{itemize}
\item a SAP group, if $G_\bz \simeq G_{\Gamma^f}\rtimes_{\delta_{\bz_0}} \Delta$, 
\item a superpythagorean group, if $G_\bz \simeq G_{\Gamma^c} \rtimes_{\delta_{\bz_0}} \Delta$. 
\end{itemize}
\end{defi}

The previous definition is motivated by Definition~\ref{SAP and superpythagorean}. As we see in the next section, these groups play a fundamental role in the study of the RPF fields. Let us recall that SAP is the acronym for Strong Approximation Property. We refer to~\cite[pages $264$ and~$265$]{lam2005introduction} for further details.

\section{Pythagorean fields}

Let~$K$ be a RPF field and $L\coloneq K(\sqrt{-1})$. We consider $G_K$ (resp. $G_L$) the maximal pro-$2$ quotient of the absolute Galois group of~$K$ (resp.~$L$). We denote by $K^\times$ (resp.~$K^{\times 2}$) the group of invertible elements (resp. invertible squares) of~$K$, and $|K^\times/K^{\times 2}|$ the cardinality of the~$\F_2$-vector space~$K^\times/K^{\times 2}$. 
The main goal of this section is to show the following result:
\begin{theo}\label{mainpyt}
Let $K$ be a field. We have the following equivalence:
\begin{itemize}
    \item[$(i)$] there exists a $\Delta$-RAAG $G_\bz$ such that $G_{\bz}\simeq G_K$,
    \item[$(ii)$] the field $K$ is RPF.
\end{itemize}
Furthermore, if $K$ is a RPF field, then the group~$G_K$ is uniquely determined by an underlying graph~$\Gamma$ with~$d_\Gamma$ vertices and an element $\bz$ in $G_\Gamma^{d_\Gamma}$.


\end{theo}

We introduce several results before proving Theorem~\ref{mainpyt}.

\subsection{Semi-direct product and the class~$\PP$}
 First, we show that we can write~$G_K$ as a semi-direct product of $G_L$ by $\Delta$.  For this purpose, let us recall \cite[Proposition]{minac1986galois}:
\begin{prop}\label{actioninv}
There exists a unique morphism $\phi \colon G_K \to \Delta$ such that $\phi(\sigma)=x_0$ for every involution $\sigma$ in $G_K$. Furthermore, $\ker(\phi)=G_L$.
\end{prop}

To conclude that $G_K$ is a semi-direct product, it is sufficient to show that the morphism $\phi$ defined in Proposition \ref{actioninv} admits a section. For this purpose, we will use \cite[Proposition $4.9$]{Minac}. Before, let us recall the following definition:

\begin{defi}[SAP and superpythagorean fields]\label{SAP and superpythagorean}
We say that $K$ is
\begin{itemize}
\item a SAP field if $G_K$ is a SAP group, i.e. $G_K\simeq G_{\Gamma^f}\rtimes_{\delta_{\bz_0}} \Delta$, 
\item a superpythagorean field if $G_K$ is a superpythagorean field, i.e.\ $G_K\simeq G_{\Gamma^c} \rtimes_{\delta_{\bz_0}} \Delta$. 
\end{itemize}
\end{defi}

\begin{exem}
We give here a few examples of RPF fields. We refer to \cite[Chapter $8$, section~$4$]{lam2005introduction} for more examples and further details.

$\bullet$ As a first example, we can take $K\coloneq \RR$. Then $\RR$ is a RPF field and both a SAP and a superpythagorean field. Indeed, we have $G_{\RR}\simeq G_\emptyset \rtimes_{\delta_{\bz_0}} \Delta$, where $\emptyset$ is the empty graph.

$\bullet$ As a second example, let us take $K\coloneq \RR((x_1)) \dots ((x_d))$, the field of iterated Laurent series over $K$. Then, the field $K$ is RPF and superpythagorean, i.e.\ we have~$G_K\simeq G_{\Gamma^c}\rtimes_{\delta_{\bz_0}} \Delta$.
\end{exem}

Let us recall \cite[Corollaries $4.2$ and $4.11$]{Minac}:
\begin{prop}\label{actioninv2}
For every integer $d_\Gamma \geq 0$, there exists two RPF fields $N$ and $C$ such that:~$(i)$~$|C^{\times}/C^{\times 2}|=|N^\times/N^{\times 2}|=2^{d_\Gamma+1}$,~$(ii)$ the field~$N$ is SAP, and~$(iii)$ the field~$C$ is superpythagorean. 

Furthermore, we fix $K$ a RPF field with $|K^{\times}/K^{\times 2}|=2^{d_\Gamma +1}$. Then there exists $N$ and $C$ as before, such that the following commutative diagram, with exact rows, holds:

\centering{
\begin{tikzcd}
1 \arrow[rr] &  & F(d_\Gamma)\simeq G_{\Gamma^f} \arrow[rr] \arrow[d, two heads] &  & G_N \arrow[rr, "\phi_N"] \arrow[d, "\pi_N", two heads] &  & \Delta \arrow[rr] \arrow[d, "\rm{id}"] &  & 1 \\
1 \arrow[rr] &  & G_L \arrow[rr] \arrow[d, two heads] &  & G_K \arrow[rr, "\phi_K"] \arrow[d, "\pi_C", two heads] &  & \Delta \arrow[rr] \arrow[d, "\rm{id}"] &  & 1    \\
1 \arrow[rr] &  & \Z_2^{d_\Gamma}\simeq G_{\Gamma^c} \arrow[rr]                    &  & G_C \arrow[rr, "\phi_C"]                                &  & \Delta \arrow[rr]            &  & 1
\end{tikzcd}}

 \justifying
In addition, the columns are all epimorphisms, and the first and last rows split.
\end{prop}

We are now able to prove the following result.

\begin{theo}\label{semidirect}
There exists a unique morphism $\phi_K\colon G_K\to \Delta$ such that:
\begin{itemize}
\item[$(i)$]
for every involution $\sigma\in G_K$, we have $\phi_K(\sigma)=x_0$,
\item[$(ii)$] the kernel of $\phi_K$ is $G_L$,
\item[$(iii)$] the morphism $\phi_K$ admits a section $\psi_K$.
\end{itemize}
\end{theo}

\begin{proof}
By Proposition \ref{actioninv}, the morphism~$\phi_K$ satisfies $(i)$ and $(ii)$. Let us show that~$\phi_K$ satisfies $(iii)$. The map $\phi_N$ introduced in Proposition \ref{actioninv2} admits a section $\psi_N$. We define~$\psi_K\coloneq \pi_N \circ \psi_N$. From Proposition~\ref{actioninv2} we have~$\phi_K \circ \psi_K(x_0)=\phi_K \circ \pi_N \circ \psi_N(x_0)=x_0$.
Thus~$\psi_K$ is a section of~$\phi_K$.
\end{proof}

Consequently, if $K$ is a RPF field, we can write:
$$G_K\coloneq G_L\rtimes_{\psi_K}\Delta.$$

\begin{coro}\label{gener}
We have $G_K/G_K^2\simeq G_L/G_L^2\times \Delta$.
\end{coro}

\begin{proof}
Let us use the notations from Proposition \ref{actioninv2}. Since $G_N$ and $G_C$ are $\Delta$-RAAGs and $G_K^2\subset G_L$, we deduce from Proposition \ref{actioninv2} that $G_K^2=G_L^2$. Consequently, 
$$G_K/G_K^2\simeq (G_K/G_L)\times (G_L/G_L^2)\simeq \Delta\times G_L/G_L^2.$$
\end{proof}

\subsection{Structural results}
In this part, we recall structural results on~$\PP$.
Let $H$ be a pro-$2$ group and let $\psi$ be a morphism of groups $\psi\colon \Delta\to {\rm Aut}(H)$. We consider the semi-direct product $G\coloneq H\rtimes_{\psi} \Delta$.

\begin{defi}[Semi-trivial action]
Let $n$ be a positive integer. We use the multiplicative notation for the product of the abelian pro-$2$ group~$\Z_2^n$. We define an action~$i_n$ of $G$ on $\Z_2^n$, that we call semi-trivial, by:
\begin{itemize}
\item[$\bullet$]for every $h\in H$ and $v\in \Z_2^n$, $i_n(h).v\coloneq v$,
\item[$\bullet$] for every $v\in \Z_2^n$, $i_n(x_0).v\coloneq v^{-1}$.
\end{itemize}
\end{defi}

\begin{rema}\label{rem semitriv}
Observe that:
$$\Z_2^n \rtimes_{i_n} G \simeq (\Z_2^n \times H) \rtimes_{i_n\times \psi} \Delta,$$
where the action $i_n\times \psi$ of $\Delta$ on $\Z_2^n \times H$ is defined by:
\begin{itemize}
\item[$\bullet$] $(i_n \times \psi)(x_0).h\coloneq \psi(h)$ for every $h\in H$,
\item[$\bullet$] $(i_n \times \psi)(x_0).v\coloneq i_n(x_0).v=v^{-1}$ for every $v\in\Z_2^n$.
\end{itemize}
\end{rema}

Let us recall \cite[Theorem]{minac1986galois} (we also refer to~\cite[§$6.2$]{minac2021koszul}):

\begin{theo}\label{structure}
The class $\PP$ is exactly the minimal class of pro-$2$ groups satisfying the following conditions:
\begin{itemize}
\item[$(i)$] the group $\Delta$ is in $\PP$,
\item[$(ii)$] if $G_1, \dots, G_m$ are in $\PP$, then $G_1\coprod G_2\coprod \dots \coprod G_m$ is also in $\PP$,
\item[$(iii)$] if $n$ is a positive integer and $G_K$ is in $\PP$, then $\Z_2^n \rtimes_{i_n} G_K$ is in $\PP$.
\end{itemize}
\end{theo}

\subsection{Structure of $G_L$}\label{struct results part2}
Theorem~\ref{semidirect} allows us to write~$G_K\simeq G_L \rtimes_{\psi_K} \Delta$. From Theorem~\ref{structure}, we study the structure of $G_L$. 

\subsubsection{Semi-direct products}
Let us consider the $\Delta$-RAAG defined by~$G_\bz\coloneq G_\Gamma \rtimes_{\delta_\bz}\Delta$. Take~$\Gamma^c$ the complete graph on $n$ vertices. We show that~$G\coloneq G_{\Gamma^c} \rtimes_{i_n}G_{\bz}$ is also~$\Delta$-RAAG.

\begin{prop}\label{rem semitriv2}
The pro-$2$ group~$G$ is~$\Delta$-RAAG. Precisely:
$$G \simeq G_{\Gamma \nabla \Gamma^c}\rtimes_{\delta_{\bz'}}\Delta,$$
where $\nabla$ denotes the join operation of graphs and $\bz'\coloneq (\bz, 1,\dots,1)$ in~$G_{\Gamma \nabla \Gamma^c}^{d_\Gamma+n}$
\end{prop}

\begin{proof}
From Remark~\ref{rem semitriv}, we infer a $(d_\Gamma+n)$-uplet $\bz'\coloneq (\bz, 1,\dots,1)$ in~$G_{\Gamma \nabla \Gamma^c}^{d_\Gamma+n}$ such that:
$$G\simeq \Z_2^n \rtimes_{i_n} (G_{\Gamma}\rtimes_{\delta_\bz }\Delta) \simeq \left(G_{\Gamma^c}\times G_\Gamma\right) \rtimes_{\delta_{\bz_0}\times \delta_\bz}\Delta\simeq G_{\Gamma \nabla \Gamma^c}\rtimes_{\delta_{\bz'}}\Delta.$$
The action $\delta_{\bz'}$ is well-defined since $\Z_2^n$ commutes with $G_\Gamma$ inside~$G_{\Gamma \nabla \Gamma^c}$.
\end{proof}

Assume there exists a group $G_{K_1}$ in $\PP$, and a positive integer~$n$ such that
$G_K\simeq \Z_2^n \rtimes_{i_n} G_{K_1}$. Let us compare $G_L$ and $G_{L_1}$:

\begin{coro}\label{kernel and semi-direct}
We have $G_{L}\simeq \Z_2^n \times G_{L_1}$. 
\end{coro}

\begin{proof}
Remark \ref{rem semitriv} gives us:
$$G_K\simeq \Z_2^n \rtimes_{i_n} \left(G_{L_1}\rtimes_{\psi_{K_1}}\Delta\right) \simeq \left( \Z_2^n \times G_{L_1}\right)\rtimes_{\psi_{K_1}\times i_n}\Delta.$$ 
From Theorem \ref{semidirect}, the group $G_{L_1}$ does not have an involution, then $G_{L_1}\times \Z_2^n$ also does not. So by Theorem \ref{semidirect}, we conclude that $G_L\simeq G_{L_1}\times \Z_2^n$.
\end{proof}

\subsubsection{Coproducts}
Here we assume that there exists two fields $K_1$ and $K_2$ such that~$G_K\simeq G_{K_1}\coprod G_{K_2}$. We study $G_L$ from $G_{L_1}$ and $G_{L_2}$.

First, we need to introduce a technical result, which will also be useful later. Let~$\Gamma_1$ and~$\Gamma_2$ be two graphs. We denote by~$\{x_{1i},\dots, x_{1d_{\Gamma_i}}\}$ a canonical system of generators of the RAAG~$G_{\Gamma_i}$ for~$i\in \{1,2\}$. We denote by~$\Gamma_0$ the graph with one vertex, and~$z$ a generator of~$G_{\Gamma_0}\simeq \Z_2$.

\begin{theo}\label{technical coprod}
The class of $\Delta$-RAAGs is stable under finite coproducts. 
\\Precisely, we have
$$\left( G_{\Gamma_1}\rtimes_{\delta_{\bz_1}} \langle x_{10}\rangle \right) \coprod \left(  G_{\Gamma_2} \rtimes_{\delta_{\bz_2}} \langle x_{20} \rangle \right) \simeq  G_{\Gamma_1 \coprod \Gamma_2 \coprod \Gamma_0} \rtimes_{\delta_{\bz}}\langle x_{10}\rangle.$$
Through the previous isomorphism, $z$ is sent to $x_{10}x_{20}$. A (canonical) system of generators of the pro-$2$ RAAG~$G_{\Gamma_1 \coprod \Gamma_2 \coprod \Gamma_0}$ is~$\{x_{11},\dots ,x_{1d_{\Gamma_1}}, x_{21},\dots, x_{2d_{\Gamma_2}}, z\}$ and the vector $\bz$ is given by:
$$\bz\coloneq (z_{11}, \dots, z_{1d_{\Gamma_1}}, z\times z_{21}, \dots, z\times z_{2d_{\Gamma_2}},1)\in (G_{\Gamma_1\coprod \Gamma_2 \coprod \Gamma_0})^{d_{\Gamma_1}+d_{\Gamma_2}+1}.$$
The vector~$\bz$ defines an action $\delta_{\bz}\colon \Delta\to G_{\Gamma}$ satisfying~\eqref{conj}.
\end{theo}

Theorem \ref{technical coprod} is a consequence of \cite[Theorem $(4.2.1)$]{NSW}, which is a profinite version of the Kurosh subgroup Theorem:

\begin{theo}
Let $G_1,\dots , G_n$ be a collection of finitely generated pro-$p$ groups. Assume that~$G=\coprod_{i=1}^n G_i$ and $H$ is an open subgroup of $G$. Define $S_i\coloneq \bigcup_{j=1}^{n_i} \{s_{i,j}\}$, where~$1\leq i \leq n$ and $n_i=|S_i|$, a system of coset representatives satisfying for every~$1\leq i \leq n$:
$$G= \bigcup_{j=1}^{n_i} G_i s_{i,j}H.$$
Then $$H=\coprod_{i=1}^n \left( \coprod_{j=1}^{n_i} G_i^{s_{i,j}}\cap H \right) \coprod F(d),$$
where $F(d)$ is a free pro-$p$ group of rank $$d=\sum_{i=1}^n \left( [G:H]-n_i\right)-[G:H]+1, \quad \text{and} \quad G_i^{s_{i,j}}=s_{i,j} G_i s_{i,j}^{-1}.$$
\end{theo}

Let us now prove Theorem \ref{technical coprod}
\begin{proof}[Proof Theorem \ref{technical coprod}]
We write $G\coloneq \left( G_{\Gamma_1}\rtimes_{\delta_{\bz_1}} \langle x_{10}\rangle \right) \coprod \left(  G_{\Gamma_2} \rtimes_{\delta_{\bz_2}} \langle x_{20} \rangle \right)$, and we define a map $\phi \colon G \to \Delta$ by:
\begin{itemize}
\item[$\bullet$] $\phi(x)=0$ for every $x\in G_{\Gamma_1}$,
\item[$\bullet$] $\phi(y)=0$ for every $y\in G_{\Gamma_2}$,
\item[$\bullet$] $\phi(x_{10})=x_{10}$ and $\phi(x_{20})=x_{10}$.
\end{itemize} 

We introduce $H\coloneq \ker(\phi)$. This subgroup is closed and of index $2$ in~$G$. So $H$ is normal and open in~$G$.
We apply Theorem \cite[Theorem $(4.2.1)$]{NSW} by taking $G_1\coloneq G_{\bz_1}$, $G_{2}\coloneq G_{\bz_2}$ and~$S_1=S_2\coloneq \{1\}$. 

The groups $G_1$ and $G_2$ are not included in $H$. Thus $G=G_1 \times S_1 \times H= G_2\times S_2 \times H$.
Note that $G\coloneq G_1\coprod G_2$. Moreover $S_1$ and $S_2$ satisfy the hypothesis of Theorem~\cite[Theorem $(4.2.1)$]{NSW}. Additionally, $G_1\cap H=G_{\Gamma_1}$ and $G_2\cap H=G_{\Gamma_2}$. Therefore from Theorem~\cite[Theorem $(4.2.1)$]{NSW}, we conclude that:
$$H\simeq G_{\Gamma_1}\coprod G_{\Gamma_2} \coprod G_{\Gamma_0} \simeq G_{\Gamma_1\coprod \Gamma_2 \coprod \Gamma_0}.$$
Through the previous isomorphism, the element~$z$ is sent to~$x_{10}x_{20}$. The map $\psi \colon \Delta \to G$, which maps $x_{10}$ to $x_{10}$ defines a section of $\phi$, and induces an action of~$\Delta\coloneq \langle x_{10}\rangle$ on~$H$. This action is precisely defined by $\delta_{\bz}$.
\end{proof}
\begin{rema}\label{technical coprod2}
More generally, from \cite[Theorem $(4.2.1)$]{NSW}, we deduce the following result. Let $\delta_1$ (resp. $\delta_2$) be an action of $\Delta\coloneq \langle x_0\rangle$ (resp. $\Delta\coloneq \langle y_0\rangle$) on a pro-$2$ group~$A$ (resp. $B$) generated by $x_i$ (resp. $y_j$).
Then
$$\left( A\rtimes_{\delta_1} \langle x_0\rangle \right) \coprod \left(  B \rtimes_{\delta_2} \langle y_0\rangle \right) \simeq \left( A\coprod B \coprod G_{\Gamma_0} \right) \rtimes_{\delta_1 \ast \delta_2}\langle x_0\rangle.$$
A generator $z$ of~$G_{\Gamma_0}$ is sent to $x_{0}y_{0}$ through the previous isomorphisms. The action $\delta_1 \ast \delta_2$ of $\Delta$ on $A \coprod B \coprod G_{\Gamma_0}$ is defined:
\begin{itemize}
\item[$\bullet$] on~$A$ by~$\delta_1$,
\item[$\bullet$] on $B$ by $\delta_2^z \colon B \to B^z;  y_i \to \delta_2(y_i)^z$, where $B^z$ is the group generated~$y_j^z$,
\item[$\bullet$] on $G_{\Gamma_0}$ by inversion.
\end{itemize}
\end{rema}

We now deduce results on the structure of $G_L$.

\begin{coro}\label{kernel and coproduct}
We have $G_L\simeq G_{L_1}\coprod G_{L_2}\coprod G_{\Gamma_0}.$
\end{coro}

\begin{proof}
Theorem~\ref{semidirect} gives us~$G_K\simeq G_L\rtimes_{\psi_K} \Delta$. Remark \ref{technical coprod2} allows us to conclude that~$G_K\simeq \left( G_{L_1} \coprod G_{L_2} \coprod G_{\Gamma_0} \right)\rtimes_{\psi_{K_1}\ast \psi_{K_2}}\Delta$. We observe that~$G_{L_1}$, $G_{L_2}$ and $G_{\Gamma_0}$ do not contain involutions. Thus~$G_{L_1} \coprod G_{L_2} \coprod G_{\Gamma_0}$ does not contain involutions.  Therefore, we conclude from Theorem \ref{semidirect} that~$G_{L}\simeq G_{L_1} \coprod G_{L_2} \coprod G_{\Gamma_0}$. 
\end{proof}

\subsection{Conclusion}
We are now able to prove~Theorem \ref{mainpyt}.


\begin{proof}[Proof Theorem \ref{mainpyt}]
We show~$(ii)$ implies~$(i)$.
If $G_K$ is $\Delta$-RAAG then, from Theorem~\ref{semidirect}, the actions $\psi_K$ and~$\delta_\bz$ coincide. Indeed $G_{\Gamma}$ does not have an involution, so $G_L=G_\Gamma$ also does not. Proposition~\ref{rem semitriv2} and Theorem \ref{technical coprod} show that the class of $\Delta$-RAAGs satisfies the conditions of Theorem \ref{structure}. Thus, by Theorem \ref{structure}, we conclude that $\PP$ is a subclass of~$\Delta$-RAAGs.

Conversely, let us show~$(i)$ implies~$(ii)$. Assume that there exists a graph $\Gamma$ and an element $\bz$ in $G_\Gamma^{d_\Gamma}$ such that~$G_K\simeq G_{\bz}$. Let $K^{(3)}$ be the compositum of all Galois extensions of degree dividing~$4$ of~$K$. From \cite[Proposition $2.1$]{minac1996witt}, we infer that $G_{K,3}=G_{K^{(3)}}$. From \cite[Theorems~$2.7$ and~$2.11$]{minac1990formally} and Corollary \ref{Delta and Pyt}, we deduce that $K$ is RPF.

The last part of our result is given by~Corollary~\ref{graph iso}.
\end{proof}

\begin{exem}\label{intermexam}
Let us give a few examples.

$\bullet$ For every integer $d_\Gamma$, the groups $G_K\coloneq G_{\Gamma^c} \rtimes_{\delta_0}\Delta$ (superpythagorean) and $G_K\coloneq G_{\Gamma^f}\rtimes_{\delta_0} \Delta$ (SAP) are in $\PP$.

$\bullet$ From Theorem \ref{structure}, the group $G_{\bz_0}$ is in $\PP$ for
every graph $\Gamma$ with at most $3$ vertices.

$\bullet$ The following graph $\Gamma$, described by four vertices and two disjoint edges, gives the group $G\coloneq G_{\Gamma}\rtimes_{\delta_{\bz_0}}\Delta \simeq \left( \Z_2^2\coprod \Z_2^2\right) \rtimes_{\delta_{\bz_0}}\Delta$:

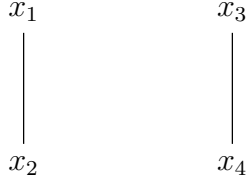
\centering{
\begin{tikzcd}
x_1 \arrow[dd, no head] &  & x_3 \arrow[dd, no head] \\
                            &  &                             \\
x_2                     &  & x_4                    
\end{tikzcd}}

\justifying
From the results in the subsection \ref{struct results part2}, the previous group is not in $\PP$. However, \cite[Theorem $2$]{snopce2022right} shows that the group $G_{\Gamma}$ is realizable as a group $G_L$, for some field~$L$ containing a square root of $-1$.
\end{exem}

\begin{rema}
If $K$ is a RPF field, then the group~$G_K$ is uniquely determined by an underlying graph~$\Gamma$ and an element $\bz$ in $G_\Gamma^{d_\Gamma}$.
In the next section, we prove a result stronger than Theorem~\ref{mainpyt}: Theorem~\ref{Graph and Pyt}. This Theorem shows that the element~$\bz$ in~$G_\Gamma^{d_\Gamma}$ is uniquely determined by the graph $\Gamma$ for RPF fields. 
\end{rema}

\section{$\Delta$-RAAG theory and Pythagorean fields}
This section investigates~$Gr(\PP)$. This is the class of graphs which are underlying graphs of groups in~$\PP$. Let~$\Gamma$ be in $Gr(\PP)$. From~\cite[Theorem $1.2$]{snopce2022right}, the graph~$\Gamma$ does not contain as subgraphs the square graph~$C_4$ and the following line~$L_3$:

\centering

\begin{tikzcd}
i_1 \arrow[rr, no head] &  & i_2 \arrow[rr, no head] &  & i_3 \arrow[rr, no head] &  & i_4
\end{tikzcd}

\justifying 
It is also shown that $H^\bullet(G_\Gamma)$ is (universally) Koszul, and $G_\Gamma$ is absolutely torsion-free. However, this condition is not sufficient to characterize~$Gr(\PP)$. We refer to Example~\ref{intermexam}.


\subsection{Structure of~$Gr(\PP)$}
We start this section by giving examples of graphs in~$Gr(\PP)$.

\begin{coro}
Let $(n,m)$ be a couple of positive integers. We define $\Gamma^{cn}$ the complete graph on $n$ vertices, and $\Gamma^{fm}$ the free graph on $m$ vertices. Then~$\Gamma^{cn} \coprod \Gamma^{fm}$ and~$\Gamma^{cn} \nabla \Gamma^{fm}$ are both in~$Gr(\PP)$.

Precisely, the following groups are realizable as groups in $\PP$:
\begin{itemize}
\item[$(i)$] $G_{(\Gamma^{cn} \coprod \Gamma^{fm}), \bz_0}$, 
\item[$(ii)$] $G_{(\Gamma^{cn} \nabla \Gamma^{fm}), \bz_0}$, where $\nabla$ denotes the join of two graphs.
\end{itemize}
\end{coro}

\begin{proof}
The groups $G_{\Gamma^{cn}, \bz_0}$ and $G_{\Gamma^{fm}, \bz_0}\simeq \Delta\coprod \dots \coprod \Delta$ are realizable as pro-$2$ quotients of absolute Galois groups over superpythagorean and SAP fields. Observe from Theorem~\ref{technical coprod} and Proposition~\ref{rem semitriv2} that:
$$G_{(\Gamma^{cn} \coprod \Gamma^{fm}), \bz_0}\simeq \Delta \coprod \dots \coprod \Delta \coprod G_{\Gamma^{cn}, \bz_0 } , \quad \text{and} \quad G_{(\Gamma^{cn} \nabla \Gamma^{fm}), \bz_0}=(G_{\Gamma^{cn}}\times G_{\Gamma^{fm}}) \rtimes_{i_n\times \delta_{\bz_0'}} \Delta,$$
with~$\bz_0'\coloneq (1,\dots, 1)\in G_{\Gamma^{fm}}^m$.
We conclude using Theorem \ref{structure}.
\end{proof}

We now study the structure of~$Gr(\PP)$. Consider $\Gamma$ in~$Gr(\PP)$. We denote by $m_0$ and~$m$ the number of isolated vertices and nontrivial connected components of $\Gamma$. We introduce~$\Gamma_i$, the nontrivial connected components of~$\Gamma$, with $1\leq i \leq m$. Then from Corollary~\ref{kernel and coproduct}, we can write:
$$\Gamma\coloneq \coprod_{i=1}^m \Gamma_i \coprod_{j=1}^{m_0}\Gamma_0,$$
with~$\Gamma_0$ a graph with one vertex. This gives us:
$G_\Gamma \simeq \coprod_{i=1}^m G_{\Gamma_i} \coprod_{j=1}^{m_0}G_{\Gamma_0}.$

\begin{lemm}\label{disconnected graph}
Assume that~$\Gamma$ is in $Gr(\PP)$. Then for every~$1\leq i \leq l$, the graph~$\Gamma_i$ is in~$Gr(\PP)$. Additionally~$m_0\geq m-1$. 

Concretely, there exists RPF fields $K_i$ with underlying graphs~$\Gamma_i$, for $1\leq i \leq m$, such that:
$$G_K\simeq \coprod_{i=1}^m G_{K_i}\coprod_{j=1}^{m_0-m+1}G_{\Gamma_0},$$
where $K$ is in~$\PP$ and has underlying graph~$\Gamma$.
\end{lemm}

\begin{proof}
Assume that~$\Gamma$ is disconnected. Then from Theorem~\ref{structure} and Corollary~\ref{kernel and coproduct}, there exist two RPF fields~$K_a$ and~$K_b$ with underlying graphs~${\Gamma_a}$ and~${\Gamma_b}$ such that:
$$G_K\simeq \G_{K_a}\coprod G_{K_b}, \quad G_\Gamma\simeq G_{\Gamma_a}\coprod G_{\Gamma_b}\coprod G_{\Gamma_0}.$$
From Lemma~\ref{Graph identification}, we can write~$\Gamma\simeq \Gamma_a\coprod \Gamma_b\coprod \Gamma_0$. Applying the previous argument to $\Gamma_a$ and $\Gamma_b$, we can find RPF fields~$K_{u_1},\dots, K_{u_{m'}}$ with underlying connected graphs $\Gamma_{u_1},\dots, \Gamma_{u_{m'}}$ such that:
$$G_K\simeq \coprod_{i=1}^{m'} G_{K_{u_i}} \quad G_L=\coprod_{i=1}^{m'} G_{\Gamma_{u_i}}\coprod_{j=1}^{m'-1}G_{\Gamma_0}.$$
Applying~Lemma~\ref{Graph identification}, this gives us the decomposition:
$$\Gamma \simeq \coprod_{i=1}^{m'} \Gamma_{u_i}\coprod_{j=1}^{m'-1}\Gamma_0.$$
We observe that $m$ is the number of graphs $\Gamma_{u_i}$ with more than two vertices. So~$m'\geq m$. Thus~$m_0\coloneq 2(m'-m)+m-1\geq m-1$. This allows us to conclude.
\end{proof}

\begin{lemm}\label{connected graph}
Assume that $\Gamma$ is connected and in~$Gr(\PP)$. Then there exists an integer~$u_\Gamma$ and a RPF field $K'$ with a unique disconnected underlying graph~$\Gamma'$ such that:
$$G_K\simeq G_{\Gamma^{cu_\Gamma}}\rtimes_{i_{u_\Gamma}}G_{K'},$$
where~$\Gamma^{cu_\Gamma}$ is the complete graph on~$u_\Gamma$ vertices.
Consequently, the underlying graph of~$G_K$ is~$\Gamma\coloneq \Gamma^{cu_\Gamma}\nabla \Gamma'$.
\end{lemm}

\begin{proof}
From Proposition~\ref{rem semitriv2} and Corollary~\ref{kernel and semi-direct}, there exists a RPF field~$K_a$ with underlying field~${\Gamma_a}$ such that:
$$G_K\simeq G_{\Gamma_0}\rtimes_i G_{K_a}, \quad G_\Gamma\simeq G_{\Gamma_0}\times G_{\Gamma_a}.$$
Applying the same argument to~$G_{K_a}$. We infer an integer~$u_\Gamma$ and a RPF field~$K'$ with disconnected underlying graph~$\Gamma'$ such that:
$$G_K\simeq G_{\Gamma^{cu_\Gamma}}\rtimes_{i_{u_\Gamma}}G_{K'}, \quad G_{\Gamma}\simeq G_{\Gamma^{cu_\Gamma}}\times G_{\Gamma'}.$$
From Lemma~\ref{Graph identification}, we infer $\Gamma\simeq \Gamma^{cu_\Gamma}\nabla \Gamma'.$
\end{proof}

From Lemmata~\ref{connected graph} and~\ref{disconnected graph}, we can show the converse of Theorem~\ref{mainpyt}. This allows us to conclude the proof of Theorem~\ref{Theo fonda}.
\begin{theo}\label{Graph and Pyt}


Let us consider two RPF fields $K_1$ and $K_2$, with underlying graphs~$\Gamma_1$ and~$\Gamma_2$. If $\Gamma_1\simeq \Gamma_2$, 
then~$G_{K_1}\simeq G_{K_2}$. Consequently, $\Delta$-RAAGs in~$\PP$ are uniquely determined by their underlying graphs.
\end{theo}

\begin{proof}
We define $\Gamma$ the underlying graph of $G_{K_1}$ and $G_{K_2}$. We show by induction on~$d_\Gamma$ that $G_{K_1}\simeq G_{K_2}$.

If $d_\Gamma=0$, then $G_{K_1}\simeq G_{K_2}\simeq \Delta$.

If $d_\Gamma=1$, then $G_{K_1}\simeq G_{K_2}\simeq G_{\Gamma_0}\rtimes_{\bz_0}\Delta$. Both groups are SAP and superpythagorean.

If $d_\Gamma>1$, we distinguish two cases:

$\bullet$ Assume that the graph $\Gamma$ is connected. Then from~Lemma~\ref{connected graph}, there exists an integer~$u_\Gamma>0$ and two RPF fields~$K_1'$ and $K_2'$ with same disconnected underlying graph~$\Gamma'$ such that:
$$G_{K_1}\simeq G_{\Gamma^{cu_\Gamma}}\rtimes_{i_{u_\Gamma}}G_{K_1'}, \quad G_{K_2}\simeq G_{\Gamma^{cu_\Gamma}}\rtimes_{i_{u_\Gamma}}G_{K_2'}.$$ 
Since~$d_{\Gamma'}<d_{\Gamma}$, we deduce, by induction hypothesis, that~$G_{K_1'}\simeq G_{K_2'}$. Thus~$G_{K_1}\simeq G_{K_2}$.

$\bullet$ If we are not in the previous case, then the graph $\Gamma$ is disconnected.
Following Lemma~\ref{disconnected graph}, there exist RPF fields~$K_{1i}$ and~$K_{2i}$ with underlying graphs~$\Gamma_i$, for~$1\leq i \leq m$ such that:
$$G_{K_1}\simeq \coprod_{i=1}^mG_{K_{1i}}\coprod_{j=1}^{m_0-m-1}G_{\Gamma_0}, \quad \text{and} \quad  G_{K_2}\simeq \coprod_{i=1}^mG_{K_{2i}}\coprod_{j=1}^{m_0-m-1}G_{\Gamma_0}.$$ 
Since for every~$1\leq i \leq m$ we have $d_{\Gamma_i}<d_{\Gamma}$, we deduce by induction hypothesis that~$G_{K_{1i}}\simeq G_{K_{2i}}$. Thus $G_{K_1}\simeq G_{K_2}$.
\end{proof}

\begin{exem}\label{second set of examples}
Let us consider the groups 
$$G_1\coloneq (\Z_2^2\rtimes_{\delta_{\bz_0}}\Delta) \coprod (\Z_2^2\rtimes_{\delta_{\bz_0}}\Delta), \quad \text{and} \quad G_2\coloneq (\Z_2\rtimes_i (F(2)\rtimes_{\delta_{\bz_0}} \Delta))\coprod \Delta\coprod \Delta.$$
These groups are $\Delta$-RAAGs, and in $\PP$.
We can represent them with the following graphs:

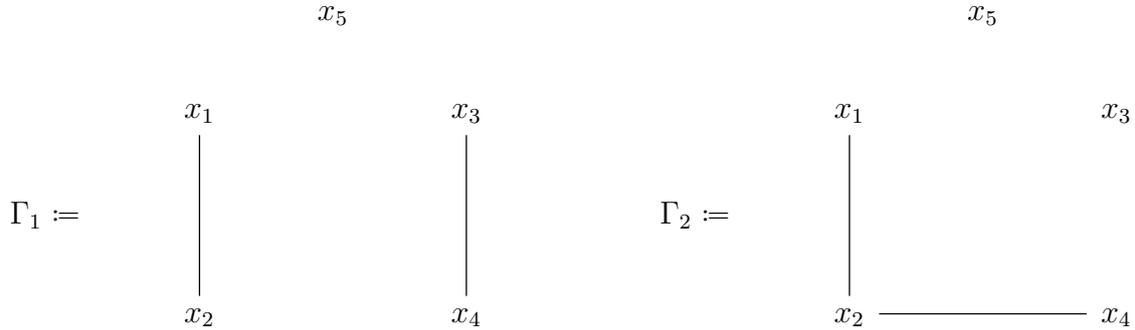
\begin{tikzcd}
           &                         & x_5 &                         &  &            &                         & x_5 &     \\
           & x_1 \arrow[dd, no head] &     & x_3 \arrow[dd, no head] &  &            & x_1 \arrow[dd, no head] &     & x_3 \\
\Gamma_1\coloneq  &                         &     &                         &  & \Gamma_2\coloneq  &                         &     &     \\
           & x_2                     &     & x_4                     &  &            & x_2 \arrow[rr, no head] &     & x_4
\end{tikzcd}

\justifying
Precisely, if we define 
$$\bz_1\coloneq (1,1,x_5,x_5,1)\in G_{\Gamma_1}^5, \quad \text{and} \quad \bz_0\coloneq (1,1,1,1,1)\in G_{\Gamma_2}^5,\quad \text{and} \quad \bz_0'\coloneq (1,1,1,1,1)\in G_{\Gamma_1}^5.$$
we infer $G_1\coloneq G_{\Gamma_1}\rtimes_{\delta_{\bz_1}}\Delta$ and $G_2\coloneq G_{\Gamma_2}\rtimes_{\delta_{\bz_0}}\Delta$. 
We observe that $G_1$ and $G_2$ have the same Poincaré and gocha series, but are not isomorphic, since $\Gamma_1$ and $\Gamma_2$ are not. We refer to Theorem~\ref{mainpyt}.

Let us now consider the group~$G_3\coloneq G_{\Gamma_1}\rtimes_{\delta_{\bz_0'}}\Delta$. We observe that~$G_1$ and~$G_3$ are not isomorphic. Thus from Theorem~\ref{Graph and Pyt}, the group~$G_3$ is not in~$\PP$. But of course $G_3$ has the same gocha series as~$G_1$ and~$G_2$. 
\end{exem}

\subsection{Koszulity, cohomology and gocha series for groups in $\PP$}
Let~$\Gamma$ be in~$Gr(\PP)$. In this subsection we study~$\Gamma(t)$. Let us also refer to~\cite{weigel2015graded} for a study of clique polynomials in other contexts.

We start this subsection by answering to a question asked by Weigel in \cite{weigel652koszul} for the class $\PP$. \textit{If~$G$ is in~$\PP$, is~$\E(G)$ Koszul?}
\begin{prop}\label{koszulpyt}
If $G$ is in $\PP$, then $\E(G)$ is Koszul. 
Consequently, there exists a graph~$\Gamma$ and a vector $\bz\in G_\Gamma^{d_\Gamma}$ such that~$G\simeq G_{\bz}$. Furthermore, the presentation~\eqref{z-Pres} is minimal and 
$$d(G)\coloneq \dim_{\F_2} H^1(G)=d_\Gamma+1 \quad \text{and} \quad r(G)\coloneq \dim_{\F_2} H^2(G)=d_\Gamma+r_\Gamma +1.$$
Additionally, the Poincaré series of~$G$ is given by:
$$H^\bullet(G,t)\coloneq \sum_n \dim_{\F_2} H^n(G)t^n=\frac{\Gamma(t)}{1-t}.$$
\end{prop}

\begin{proof}
Take $G$ in $\PP$.  From \cite[Theorem $F, (3)$]{minac2020enhanced} or \cite[Theorem $D, (g)$]{minac2021koszul}, we observe that $H^\bullet(G)$ is Koszul. From Theorem~\ref{mainpyt}, there exists $\bz$ and~$\Gamma$ such that~$G_{\bz}\simeq G$. Then from Remark~\ref{quad defined} and~\cite[Theorem $F$]{minac2021koszul}, we conclude that $H^\bullet(G)$ is the quadratic dual of $\E(G)$. Consequently $\E(G)$ is also Koszul. We infer from Corollary~\ref{gocha Delta Raag2} that:
\begin{multline*}
H^\bullet(G,t)\coloneq \sum_n \dim_{\F_2} H^n(G)t^n=\frac{\Gamma(t)}{1-t}
\\=(1+d_\Gamma t+r_{\Gamma}t^2+\dots)(1+t+t^2+\dots)=1+(d_\Gamma +1)t+(r_\Gamma +d_{\Gamma}+1)t^2+\dots
\end{multline*}
Thus $d(G)=d_\Gamma+1$, and $r(G)=d_\Gamma+r_\Gamma +1$.
\end{proof}

We conclude this subsection by using results from \cite{minachilbseries} to characterise which polynomials~$\Gamma(t)$ are realizable, for~$\Gamma$ in~$Gr(\PP)$.
\begin{theo}\label{proving}
Let~$\Gamma$ be an undirected graph. The graph~$\Gamma$ is in~$Gr(\PP)$, if and only if there exists  integers~$s$ and~$a_0,\dots , a_{s-1}$ satisfying all of the following conditions:
 
    $(i)$ $0\leq a_0,\dots ,a_{s-2}$, 
    
    $(ii)$ $1\leq a_{s-1}, s$, 
    
    $(iii)$~$a_0+\dots+a_{s-1}+s=d_\Gamma$,
    
     $(iv)$ $\Gamma(t)= (1+t)^{s-1}+t\left( (1+t)^{s-1}a_{s-1}+(1+t)^{s-2}a_{s-2}+\dots +a_0\right).$
\end{theo}

\begin{proof}
    The Poincaré series of $G_\Gamma$ is given by:
    $$H^\bullet(G_\Gamma,t)\coloneq \sum_{n\in \NN} \dim_{\F_2}H^n(G_\Gamma)t^n= \Gamma(t).$$
Let~$\Gamma$ be in~$Gr(\PP)$.    From~\cite[Theorem $11$]{minachilbseries}, we conclude that $\Gamma(t)= (1+t)^{s-1}+t\left( (1+t)^{s-1}a_{s-1}+\dots +a_0\right)$ for some integers $s$, $a_0,\dots , a_{s-1}$ satisfying $(i)-(iii)$.
    
The converse is a consequence of \cite[Theorem $11$]{minachilbseries}.
\end{proof}

\section{An interesting example}\label{final example}
We conclude this paper with an example of a pro-$2$ group which is not a maximal pro-$2$ quotient of an absolute Galois group, but satisfies the Koszul, the strong Massey Vanishing and the Kernel Unipotent properties. Let us first recall these properties for a general pro-$p$ group~$G$.

$\bullet$ We say that a pro-$p$ group $G$ satisfies the Koszul property if the cohomology ring~$H^\bullet(G;\F_p)$ is Koszul. Positselski~\cite{positselski2014galois}
formulated the following conjecture. All maximal~pro-$p$ quotients of absolute Galois groups over fields, containing the primitive $p$-th roots of unity, satisfy the Koszul property.

$\bullet$ We recall the Kernel Unipotent property. We fix a positive integer $n$, and define:
$$G_{\langle n \rangle}\coloneq \bigcap_{\rho\colon G \to \mathbb{U}_n(\F_p)} {\rm ker}\rho.$$
We easily observe that $G_n \subset G_{\langle n\rangle}$, and we say that $G$ satisfies the $n$-Kernel Unipotent property if $G_{\langle n\rangle}=G_n$. If $G$ satisfies the $n$-Kernel Unipotent property for every $n$, we say that~$G$ satisfies the Kernel Unipotent property. Efrat and the last two authors \cite[Appendix]{minavc2015kernel} gave, for every integer $n\geq 3$, examples of groups which are not maximal pro-$p$ quotients of absolute Galois groups and do not satisfy the $n$-Kernel Unipotent property. The last two authors also formulated the following conjecture \cite[Conjecture $1.3$]{minavc2015kernel}. All maximal pro-$p$ quotients of absolute Galois groups over fields, containing the primitive $p$-th root of unity, satisfy the Kernel Unipotent property.

$\bullet$ We recall the strong Massey Vanishing property. Let $n$ be a positive integer. Let $\alpha$ be a family of characters~$\alpha\coloneq \{\alpha_i\colon G \to \F_p\}_{i=1}^n$ in $H^1(G;\F_p)^n$ checking $\alpha_i\cup \alpha_{i+1}=0$. We say that a group $G$ satisfies the strong Massey Vanishing property if for every~$n$ and~$\alpha$, there exists a morphism $\rho\colon G \to \mathbb{U}_{n+1}(\F_p)$ such that for every~$g$ in $G$, we have
$$\rho_{i,i+1}(g)\coloneq \rho(g)_{i,i+1}= \alpha_i(g).$$ The strong Massey Vanishing property does not hold for all maximal pro-$p$ quotients of absolute Galois groups of fields containing a $p$-th root of unity. We refer to Harpaz-Wittenberg's counterexample~\cite[Example $A.15$]{guillot2018four}. However, Ramakrishna and the last three authors \cite[Theorem $1$]{maire2024strong} showed this property for maximal pro-$p$ (tame) quotients of absolute Galois groups of number fields, which do not contain a primitive $p$-th root of unity. 
Quadrelli \cite{quadrelli2023massey}-\cite{quadrelli2024massey} showed the strong Massey Vanishing property for Elementary pro-$p$ groups and the class $\PP$. Observe also that the strong Massey Vanishing property implies the Massey Vanishing property stated by the last two authors. We refer to~\cite{merkurjev2024lectures} for further details.

This section aims to prove~Theorem~\ref{impo exam}. Let us define the pro-$2$ group~$G$ by the presentation:
\begin{multline} \tag{$\bz_0$-Pres} \label{pres}
G\coloneq \langle x_0, x_1, x_2,x_3,x_4| \quad \lbrack x_1,x_2\rbrack =\lbrack x_2,x_3\rbrack =\lbrack x_3,x_4\rbrack =\lbrack x_4,x_1\rbrack =1, 
\\x_0^2=1, x_0x_jx_0x_j=1, \forall j\in [\![1;4]\!]\rangle
\end{multline}
\begin{theo}\label{counterexample}
The group $G$ is not a maximal pro-$2$ quotient of an absolute Galois group but satisfies the Koszul, the Kernel Unipotent and the strong Massey Vanishing properties. Furthermore, the presentation \eqref{pres} is minimal and $\E(G)$ is Koszul.
\end{theo}

We introduce several pro-$2$ groups which will play a fundamental role in the study of~$G$ and the proof of~Theorem~\ref{counterexample}.
 Let~$G_{13}$ and~$G_{24}$ be the subgroups of $G$ generated by~$\{x_0, x_1, x_3\}$ and~$\{x_0, x_2, x_4\}$. We also introduce~$F_{13}$ and~$F_{24}$ the subgroups of~$G$ generated by~$\{x_1,x_3\}$ and~$\{x_2,x_4\}$. We define $G_0\coloneq \langle y_1\rangle \coprod \langle y_2\rangle \coprod \langle y_3\rangle \simeq \Delta\coprod \Delta\coprod \Delta$, where $y_i^2=1$, for~$1\leq i \leq 3$. The group~$G_0$ is a quotient of~$G$ and has an easy structure. It is SAP, so in~$\PP$. 
 
Furthermore, we easily observe that we have natural surjections~$\pi_{13}\colon G_0\to G_{13}$ and~$\pi_{24}\colon G_0\to G_{24}$ defined by:
$$\pi_{13}(y_1y_2) \coloneq x_1, \quad \pi_{13}(y_1y_3)\coloneq x_3, \pi_{24}(y_1y_2) \coloneq x_2, \quad \pi_{24}(y_1y_3)\coloneq x_4 \quad \pi_{13}(y_1)=\pi_{24}(y_1)\coloneq x_0.$$

\subsection{First properties of the group~$G$}
We observe that~$G$ is a $\Delta$-RAAG given by~$G\coloneq \left( F(2)\times F(2)\right)\rtimes_{\delta_{\bz_0}}\Delta$.
The underlying graph of the $\Delta$-RAAG $G$ is the square graph $\Gamma=C_4$:

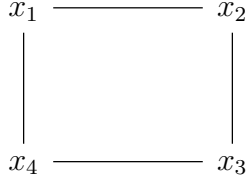
\centering{
\begin{tikzcd}
x_1 \arrow[dd, no head] \arrow[rr, no head] &  & x_2 \arrow[dd, no head]  \\
                            &  &                             \\
x_4   \arrow[rr, no head]                  &  & x_3                    
\end{tikzcd}} 

\justifying
And the action is defined by the vector $\bz_0\coloneq (1,1,1,1)$ in~$G_\Gamma^4$. Thanks to the~$\Delta$-RAAG theory, we infer the following result:

\begin{lemm}\label{not absolute quotient}
The group $G$ is not a maximal pro-$2$ quotient of an absolute Galois group.
\end{lemm}

\begin{proof}
The group $G_\Gamma$ is not realizable as a maximal pro-$2$ quotient of an absolute Galois group of a field containing a second root of unity. We refer to~\cite[Theorem $1.2$]{snopce2022right} or~\cite[Theorem~$5.6$]{Quadrelli2014}. Thus $G$ is not in $\PP$. Alternatively, we observe from Lemmata~\ref{connected graph} and~\ref{disconnected graph} that $\Gamma$ is not in~$Gr(\PP)$. Thus~$G$ is not in~$\PP$, and so not a maximal pro-$2$ quotient of an absolute Galois group. 
\end{proof}

\begin{rema}
We observe that $\Gamma$ is the only graph with clique polynomial~$\Gamma(t)\coloneq 1+4t+4t^2$. Thus from Theorem~\ref{proving}, we conclude that there exists no $\Delta$-RAAG $G'$ in~$\PP$  with underlying graph $\Gamma'$ such that $\Gamma(t)=\Gamma'(t)$.  
\end{rema}

Despite not being in $\PP$, the group $G$ contains subgroups in $\PP$.

\begin{lemm}\label{decomposition G}
The groups~$G_{13}$ and~$G_{24}$ are isomorphic to~$G_0$. The groups~$F_{13}$ and~$F_{24}$ are pro-$2$ free on $\{x_1,x_3\}$ and~$\{x_2,x_4\}$. Additionally, we have the decomposition
$$G\simeq F_{24}\rtimes_{\delta_F} G_{13}\simeq F_{13}\rtimes_{\delta_{F'}}G_{24},$$ 
where the action~$\delta_F$ and~$\delta_{F'}$ are defined by:
\begin{equation*}
\begin{aligned}
\delta_F(x_0)(x_i)=x_i^{-1}, \quad \delta_F(x_j)(x_i)= x_i, \quad \text{for } i\in\{2,4\}, \text{ and } j\in \{1,3\},
\\ \delta_{F'}(x_0)(x_j)=x_j^{-1}, \quad \delta_{F'}(x_i)(x_j)=x_j, \quad \text{for } i\in\{2,4\}, \text{ and } j\in \{1,3\}.
\end{aligned}
\end{equation*}
\end{lemm}

\begin{proof}
Let~$F(2)$ be the pro-$2$ group on two generators~$a$ and $b$. We define an action~$\psi_0$ of $G_0$ on $F(2)$ by~$\psi_0(y_i)(a)\coloneq a^{-1}$ and~$\psi_0(y_i)(b)\coloneq b^{-1}$ for~$1\leq i \leq 3$. Similarly to Remark~\ref{rem semitriv}, we observe that~$G\simeq F(2)\rtimes_{\psi_0} G_0$. Therefore, we obtain two injections~$\varphi_{13}$ and~$\varphi_{24}$ from~$G_0$ to~$G$ defined by:
$$\varphi_{ij}(y_1y_2)=x_i, \quad \varphi_{ij}(y_1y_3)=x_j, \quad \text{and} \quad \varphi_{ij}(y_1)=x_0.$$
We also introduce the natural injections $\iota_{13}\colon G_{13}\to G$ and~$\iota_{24}\colon G_{24}\to G$. We infer the following commutative diagram:

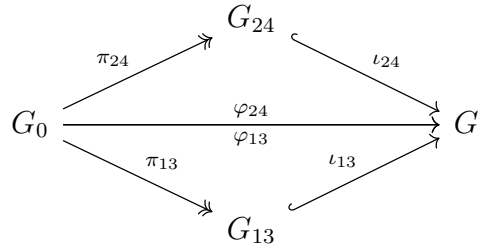
\centering{
\begin{tikzcd}
                                                                                                                                                 &  & G_{24} \arrow[rrd, "\iota_{24}", hook] &  &   \\
G_0 \arrow[rrd, "\pi_{13}", two heads] \arrow[rrrr, "\varphi_{13}"'] \arrow[rru, "\pi_{24}", two heads] \arrow[rrrr, "\varphi_{24}"] &  &                                        &  & G \\
                                                                                                                                                 &  & G_{13} \arrow[rru, "\iota_{13}", hook] &  &  
\end{tikzcd}} 

\justifying
Since $\varphi_{13}$ and $\varphi_{24}$ are both injective, the maps~$\pi_{13}$ and~$\pi_{24}$ are isomorphisms. Similarly, we show that~$F_{13}\simeq F_{24}\simeq F(2)$. This allows us to conclude.
\end{proof}
As a consequence, we infer that the subgroup of~$G$ generated by~$\{x_1,x_2,x_3,x_4\}$ is isomorphic to~$F_{13}\times F_{24}$, which is exactly~$G_\Gamma$.

\subsection{Cohomological results on~$G$}
We now study the cohomology ring~$H^\bullet(G)$. 

\begin{lemm}\label{minpres}
The algebra $\E(G)$ is Koszul. As a consequence, the presentation~\eqref{pres} is minimal and the group $G$ satisfies the Koszul property.
\end{lemm}

\begin{proof}
Let us consider $\E$ the noncommutative polynomials on~$\{X_0,X_1,X_2,X_3,X_4\}$ over~$\F_2$ where each~$X_i$ has weight $1$. A graded $\F_2$-basis of $\E$ is given by monomials, and we endow them with an order induced by~$x_0>x_1>x_3>x_2>x_4$. To show that the algebra $\E(G)$ is Koszul, it is sufficient to show that it is PBW. This property comes from Poincaré-Birkhoff-Witt, and we recall it briefly. We refer to \cite[\S $4.1$]{loday2012algebraic} or~\cite{minac2021koszul}. From Theorem \ref{gocha Delta Raag}, we have $\E(G)\simeq \E_{\bz_0}$. Thus $\E(G)$ is presented by $5$ generators $\{X_0,X_1,X_2,X_3,X_4\}$ and $9$ relations $\Rr\coloneq \{X_0^2,\lbrack X_u,X_v\rbrack; \lbrack X_0,X_i\rbrack+X_i^2\}$, with $u\in \{1,3\}$, $v\in\{2,4\}$ and $1\leq i \leq 4$. The leading monomials of these relations are given by:
$$M(\Rr)\coloneq \{X_0^2, X_1X_2, X_3X_2,X_3X_4,X_1X_4, X_0X_1,X_0X_2,X_0X_3,X_0X_4\}.$$
A monomial is said to be reduced if it does not contain any leading monomial as a submonomial. We define critical monomials by monomials of the form~$X_uX_vX_w$ where $X_uX_v$ and $X_vX_w$ are in $M(\Rr)$. In our example, the critical monomials are exactly 
$$C(\Rr)\coloneq \{X_0^2X_i, X_0X_uX_v, \quad 1\leq i \leq 4, \text{ and } u \in \{1,3\}, v\in \{2,4\}\}.$$
Each relation in $\Rr$ can be interpreted as a reduction rule which substitutes its leading monomial in $M(\Rr)$ to a sum of lower monomials. There are two ways to write a critical monomial as a sum of reduced monomials, with respect to reduction rules. This is technical; however, we can at least say that the first way is to reduce it from the left to the right, and the second way is to reduce it from the right to the left. For instance, we refer to~\cite[§2.4]{minac2021koszul} or \cite[§4.3.5]{loday2012algebraic} for further details. These reductions allow us to construct, for each critical monomial, a diagram. This diagram starts with a vertex given by a critical monomial and ends with at most two vertices. Each ending vertex is given by a sum of reduced monomials. We say that a critical monomial is confluent if its associated diagram finishes with one vertex, i.e.\ the two sums of reduced monomials coincide. The algebra $\E(G)$ is PBW if every critical monomial is confluent.

In our example, we observe that every critical monomial is confluent, so $\E(G)$ is PBW. For instance, we apply the previous reduction rules to the critical monomial~$X_0X_1X_2$. We infer the following diagram:

\center{
\begin{tikzcd}
                                      & X_0X_1X_2 \arrow[ld] \arrow[rd] &                                       \\
(X_1X_0+X_1^2)X_2 \arrow[dd]          &                                 & X_0X_2X_1 \arrow[d]                   \\
                                      &                                 & (X_2X_0+X_2^2)X_1 \arrow[d]           \\
X_1(X_2X_0+X_2^2)+X_1^2X_2 \arrow[rd] &                                 & X_2(X_1X_0+X_1^2)+X_2^2X_1 \arrow[ld] \\
                                      & X_2X_1X_0+X_2^2X_1+X_2X_1^2     &                                      
\end{tikzcd}}

\justifying 
From \cite[Proposition $1$]{hamza2023extensions} or \cite{leoni2024zassenhaus}, we deduce that the group~$G$ satisfies the Koszul property. As a consequence of Corollary~\ref{gocha Delta Raag2}, the gocha and Poincaré series of $G$ are given by:
$$gocha(G,t)\coloneq \frac{1+t}{1-4t+4t^2}, \quad \text{and} \quad H^\bullet(G,t)\coloneq \sum_n \dim_{\F_2} H^n(G)t^n= \frac{1+4t+4t^2}{1-t}.$$
Therefore the presentation \eqref{pres} is minimal.
\end{proof}

Let us denote by $\chi_0$ (resp. $\psi_i$) the character associated to $x_0$ (resp. $x_i$).
\begin{lemm}\label{cohomG}
The algebra $H^\bullet(G)$ is the quadratic dual of $\E(G)$. Concretely, the algebra~$H^\bullet(G)$ is generated by the generators $\chi_0$ and $\psi_i$ for $1\leq i\leq 4$, and relations:
$$\chi_0\cup \psi_i+\psi_i\cup \chi_0, \quad \psi_i\cup \psi_j+\psi_j\cup \psi_i, \quad \chi_0\cup \psi_i+\psi_i^2,\quad \psi_1\cup \psi_3, \quad \psi_2\cup \psi_4,$$
with $1\leq i<j\leq 4$. Thus a basis of $H^2(G)$ is given by:
$$\chi_0^2, \quad \chi_0\cup \psi_i=\psi_i^2, \quad \psi_1\cup \psi_2, \quad \psi_1\cup \psi_4, \quad \psi_3\cup \psi_2, \quad \psi_3\cup \psi_4.$$
\end{lemm}

\begin{proof}
In Lemma \ref{minpres}, we showed that the algebra $\E(G)$ is Koszul. Thus $H^\bullet(G)$ is the quadratic dual of $\E(G)$. We refer to~\cite[Proposition~$1$]{hamza2023extensions}. We compute the quadratic dual of $\E(G)$. We refer to \cite[Chapter $3$, $\S 3.2.2$]{loday2012algebraic} for further details.

By a dimension argument, we observe that the set
$$\{\chi_0\cup \psi_i+\psi_i\cup \chi_0, \quad \psi_i\cup \psi_j+\psi_j\cup \psi_i, \quad \chi_0\cup \psi_i+\psi_i^2,\quad \psi_1\cup \psi_3, \quad \psi_2\cup \psi_4\}$$
defines all relations for the quadratic dual of $\E(G)$. We conclude that the quadratic dual of $\E(G)$ has a presentation with generators $\chi_0, \psi_i$ and relations defined above.
\end{proof}

Lemma \ref{cohomG} allows us to infer:

\begin{coro}\label{cohomology structure}
We have the following isomorphisms
\begin{equation*}
\begin{aligned}
H^1(G)\simeq & H^1(G_{13})\bigoplus H^1(F_{24})\simeq H^1(G_{24})\bigoplus H^1(F_{13}),
\\ H^2(G)\simeq & H^2(G_{13})\bigoplus \left( H^1(G_{13})\cup \psi_2\right) \bigoplus \left( H^1(G_{13})\cup \psi_4\right),
\end{aligned}
\end{equation*}
where we identify 

$\bullet$ $H^1(G_{ij})$ and $H^1(F_{ij})$ with the subgroups of $H^1(G)$ generated by $\{\chi_0, \psi_i, \psi_j\}$ and~$\{\psi_i, \psi_j\}$, 

$\bullet$ $H^2(G_{13})$ with the subgroup of $H^2(G)$ generated by $\chi_0^2$, $\chi_0\cup \psi_1=\psi_1^2$ and $\chi_0\cup \psi_3=\psi_3^2$.
\end{coro}

\subsection{The Kernel Unipotent property}
In this subsection, we show that $G$ satisfies the Kernel Unipotent property. It is strongly inspired by \cite[§$2$]{minavc2015kernel}.

We begin with the following lemma:

\begin{lemm}\label{kerfree}
The group $G_0$ satisfies the Kernel Unipotent property.
\end{lemm}

\begin{proof}
For every positive integer~$n$ and every~$g\in G_{0,n-1}\setminus G_{0,n}$, we construct a map~$\rho_g\colon G_0\to \mathbb{U}_n$ such that~$\rho_g(g)\neq \mathbb{I}_n$.

Let us recall that $G_0\simeq \Delta \coprod \Delta \coprod \Delta\coloneq \langle y_1\rangle \coprod \langle y_2\rangle \coprod \langle y_3\rangle$, where $y_i$ denotes an involution. Let us define $E$ by the set of noncommutative series on $\F_2$ over $Y_1$, $Y_2$ and $Y_3$, and~$I$ the closed two-sided ideal of~$E$ generated by $\{Y_1^2, Y_2^2, Y_3^2\}$. Here, every $Y_1$, $Y_2$ and $Y_3$ have weight~$1$. From the Magnus isomorphism, we infer the isomorphism~$\psi_{G_0}\colon E(G_0)\simeq E/I$, which maps~$y_i$ to~$1+Y_i$. In particular, if~$g\in G_{0,n-1}\setminus G_{0,n}$, then we can write:
$$\psi_{G_0}(g)\coloneq 1+\sum_W \epsilon_W(g)W+ W_{\geq n},$$
for $\epsilon_W(g)$ in $\F_2$, $W\coloneq Y_{i_1}Y_{i_2}\dots Y_{i_{n-1}}$ words satisfying $i_j\neq i_{j+1}$ with $1\leq j \leq n-2$, and~$W_{\geq n+1}$ a series of degree larger than~$n+1$. The previous form is unique. 

We introduce $\mathbb{M}_{n}$, the set of $n\times n$ matrices with coefficients in $\F_2$. We denote by~$\delta_{i,j}$ the elementary $n\times n$-matrix, which is equal to zero everywhere except in $(i,j)$. Let us define a morphism $\rho_W$ from $E(G_0)$ to $\mathbb{M}_n$ satisfying $\rho_{W}(W)=\delta_{1,n}$.

We fix a word $W\coloneq Y_{i_1}Y_{i_2}\dots Y_{i_{n-1}}$ in $E(G_0)$ satisfying $i_j\neq i_{j+1}$. And for every~$1\leq i \leq 3$, we define a map $\psi_i \colon [\![1;n-1]\!]\to \F_2$ by:
$$\psi_i(j)\coloneq 1, \text{ if } Y_{i_j}=Y_i, \quad \text{else } \psi_i(j)\coloneq 0.$$
We introduce the matrix 
$$M_i\coloneq \sum_{j=1}^{n-1} \psi_i(j) \delta_{j,j+1}.$$
Since for every $j$, we have $Y_{i_j}\neq Y_{i_{j+1}}$, we observe for every $1\leq i \leq 3$ the equality $M_i^2=0$.

Thus we define a morphism $\rho_W$ by $\rho_W(1)=\mathbb{I}_{n}$ and $\rho_W(Y_i)\coloneq M_i$. If $W'\coloneq Y_{u_1}\dots Y_{u_{n-1}}$ is a word also satisfying $u_j\neq u_{j+1}$, then we have:
$$\rho_W(W')=\psi_{u_1}(1)\dots \psi_{u_{n-1}}(n-1) \delta_{1,n}.$$
In particular $\rho_W(W')=\delta_{1,n}$ only if $W'=W$, else $\rho_W(W')=0$.

Let us now fix a nontrivial element $g$ in $G_{0,n-1}/ G_{0,n}$. Then there exists a word~$W$ of degree~$n-1$ such that~$\epsilon_W(g)\neq 0$. We define~$\rho_g$ the map induced by~$\rho_W$, i.e.\ ~$\rho_g(y_i)\coloneq \mathbb{I}_{n}+M_i$.
We observe that $\rho_g(y_i)^2=\mathbb{I}_{n}^2+M_i^2=\mathbb{I}_{n}$, and $\rho_g(Y_i)=\rho_W(Y_i)$. Thus we have:
$$\rho_g(g)=\mathbb{I}_{n}+\sum_{W'}\epsilon_{W'}(g)\rho_g(W')=\mathbb{I}_{n}+\epsilon_W(g)\rho_W(W)=\mathbb{I}_{n}+\delta_{1,n}\neq \mathbb{I}_{n},$$
where the sum is indexed by all monomials $W'$ in $E$ of degree~$n$.
\end{proof}

\begin{rema}
Without loss of generality, we can expand the proof of Lemma~\ref{kerfree} and show that if $G$ is a SAP group, then $G$ satisfies the Kernel Unipotent property.
\end{rema}

We can now prove the Kernel Unipotent property for $G$.

\begin{prop}
The group $G$ satisfies the Kernel Unipotent property.
\end{prop}

\begin{proof}
From Lemma~\ref{decomposition G}, we observe that for every positive integer~$n$ we have semi-direct product decompositions:
$$G/G_n\simeq F_{24}/F_{24,n} \rtimes_{\overline{\delta_{F}}} G_{13}/G_{13,n}\simeq F_{13}/F_{13,n} \rtimes_{\overline{\delta_{F'}}} G_{24}/G_{24,n}.$$ 
Let us consider $x\coloneq uh$ with $u\in F_{24}/F_{24,n}$ and $h\in G_{13}/G_{13,n}$,  a nontrivial element in~$G/G_n$. We need to construct a map $\rho_x\colon G/G_n \to \mathbb{U}_n$ such that $\rho_x(x)\neq \mathbb{I}_n$. For this purpose, we distinguish two cases:

$(i)$ Assume that $h\neq 1$. By Lemma \ref{kerfree}, the group $G_{13}$ satisfies the Kernel Unipotent property, thus we have a map $\eta_h\colon G_{13}/G_{13,n} \to \mathbb{U}_n$ such that $\eta_h(h)\neq \mathbb{I}_n$. Let us define $\rho_x$ by:
$$\rho_x|_{G_{13}/G_{13,n}}\coloneq \eta_h, \quad \text{and} \quad \rho_x|_{F_{24}/F_{24,n}}\coloneq \mathds{1}.$$
This map is well-defined, and we have $\rho_x(x)\coloneq \rho_x(uh)=\eta_h(h)\neq \mathbb{I}_n.$

$(ii)$ Now assume that $h=1$. Then $x\coloneq u$ is in $F_{24}/F_{24,n}$, and not trivial. Additionally, Proposition~\ref{lower RAAGs} gives us the injection~ $F_{24}/F_{24,n} \subset G_{24}/G_{24,n}$. Since~$G_{24}\simeq G_0$, we apply the same argument as~$(i)$, replacing~$G_{13}$ by~$G_{24}$ and~$F_{24}$ by~$F_{13}$. 
\end{proof}

\begin{rema}\label{kernel uni product}
Expanding the previous proof, we can show that if $G_1$ and $G_2$ check the Kernel Unipotent property, then $G_1\times G_2$ also satisfies it.  

Observe that $F(2)$ also satisfies this property. We refer for instance to~\cite[Theorem $2.6$]{minavc2015kernel}. Thus the group~$G_\Gamma\simeq F_{13}\times F_{24}$ satisfies the Kernel Unipotent property. This group is not a maximal pro-$2$ quotient of an absolute Galois group. We refer to~\cite[Theorem~$1.2$]{snopce2022right}.
\end{rema}

\subsection{The strong Massey Vanishing property}\label{results on SMVP}
In this subsection, we show that~$G$ satisfies the strong Massey Vanishing property. This subsection is heavily inspired by \cite[§$4.3$]{quadrelli2024massey} and we mostly use the same notations.
Let us also recall that the structure of~$H^\bullet(G)$ is given by Lemma~\ref{cohomG} and Corollary~\ref{cohomology structure}.

\begin{lemm}\label{cupzero}
Assume that $\alpha$ and $\alpha'$ are nontrivial elements in $H^1(G)$ satisfying~$\alpha\cup \alpha'=0$. We have the following alternative. Either:
\begin{itemize}
\item[$\bullet$] $\alpha$ and $\alpha'$ are in $H^1(G_{13})$ and different from $\chi_0$,
\item[$\bullet$] $\alpha$ and $\alpha'$ are in $H^1(G_{24})$ and different from $\chi_0$,
\item[$\bullet$] $\alpha$ and $\alpha'$ are neither in $H^1(G_{13})$ nor in $H^1(G_{24})$, but they check the equality:
$$\alpha'=\alpha+\chi_0.$$
\end{itemize}
\end{lemm}

\begin{proof}
In this proof, we use the following fact coming from Corollary~\ref{cohomology structure} and~Lemma~\ref{cohomG}. If $\alpha$ is in $H^1(G_{13})$, then $\alpha\cup \psi_2=\alpha\cup \psi_4=0$ if and only if $\alpha=0$.

Let us take nontrivial elements $\alpha$ and $\alpha'$ in $H^1(G)$ such that $\alpha\cup \alpha'=0$. We write 
$$\alpha\coloneq a\chi_0+\beta+ b\psi_2+c\psi_4, \quad \text{and} \quad \alpha'\coloneq a'\chi_0+\beta'+ b'\psi_2+c'\psi_4,$$
with $\beta,\beta'$ in the vector space generated by $\psi_1$ and $\psi_3$, that is identified with~$H^1(F_{13})\subset H^1(G_{13})$. Recall that we have the decomposition:
$$H^1(G_{13})\coloneq H^1(F_{13})\bigoplus \chi_0 \F_2.$$
Using the relations from Lemma \ref{cohomG}, we get the following expression for $\alpha\cup \alpha'$:
\begin{multline*}
\alpha\cup \alpha'\coloneq  aa'\chi_0^2+a\chi_0\beta' +a'\beta \chi_0+\beta\beta' +
\\(b\beta'+b'\beta+ (ab'+a'b+bb')\chi_0)\cup \psi_2+
\\(c\beta'+c'\beta+(ac'+a'c+cc')\chi_0)\cup \psi_4.
\end{multline*}

Then, solving $\alpha\cup \alpha'=0$, we identify the last two right-hand terms to infer the following system:

\begin{equation*}
\left\{
\begin{aligned}
b\beta'+b'\beta=0
\\ab'+a'b+bb'=0
\\c\beta'+ c'\beta=0,
\\ac'+a'c+cc'=0.
\end{aligned}
\right.
\end{equation*}

We distinguish several cases:


$(a)$ If we assume $(b,b')=(1,0)$, then $\beta'=0$ and $a'=0$. So~$c'=1$. Thus either
$(i)$ $c=0$ or $(ii)$ $c=1$. In the case~$(i)$, we have $\alpha'=\psi_4$, and $\alpha=\psi_2$. In the case~$(ii)$, we have~$c=1$. Thus $\beta=0$ and $a=1$ so~$\alpha=\chi_0+\psi_2+\psi_4$, and $\alpha'=\psi_4$.

$(b)$ The case~$(b,b')=(0,1)$ is symetric to the case~$(a)$.
Consequently, in cases $(a)$ and~$(b)$, we always infer that $\alpha$ and $\alpha'$ are in $H^1(G_{24})$.

$(c)$ We assume that $(b,b')=(1,1)$. This case imposes that~$\beta=\beta'$ and $a'=a+1$. We distinguish two cases.
If~$(i)$ $\beta=0$, then $\alpha$ and $\alpha'$ are both in~$H^1(G_{24})$.
If~$(ii)$ $\beta\neq 0$, then $c=c'$, so we infer that~$\alpha'= \alpha+\chi_0$, and $\alpha$, $\alpha'$ are neither in $H^1(G_{13})$  nor in $H^1(G_{24})$.

$(d)$ We assume that~$(b,b')=(0,0)$. Since~$\alpha$ and~$\alpha'$ are nontrivial, we infer that~$c=c'$. We have two case: either~$(i)$ $c=c'=0$ or~$(ii)$ $c=c'=1$. If~$(i)$ then~$\alpha$ and~$\alpha'$ are both in~$H^1(G_{13})$. Else~$(ii)$ gives us~$\beta=\beta'$ and~$a=a'+1$.

We distinguished all cases. To recap:


$\bullet$ in the cases $(a)$, $(b)$ and $(c,i)$ we have $\alpha$ and $\alpha'$ both in $H^1(G_{24})$,

$\bullet$ in the case~$(d,i)$ we have~$\alpha$ and $\alpha'$ both in $H^1(G_{13})$,

$\bullet$ in the cases $(c,ii)$ and $(d,ii)$, we have $\alpha'=\alpha+\chi_0$, and $\alpha$ and $\alpha'$ are neither in $H^1(G_{13})$ nor in~$H^1(G_{24})$.
\end{proof}

Let us recall a result from Quadrelli \cite[Lemma $4.2$]{quadrelli2024massey}:

\begin{lemm}\label{lemmquad}
For $n > 2$, there exist matrices $A_1, A_2, B_1, B_2 \in \mathbb{U}_{n+1}$ such that: $(i)$ the $(i, i+1)$-entries of both $A_1$ and $A_2$ are equal to $1$, for $1\leq i \leq n$, $(ii)$ $B_1$ and~$B_2$ are given by

\[
B_1 =
\begin{pmatrix}
1 & 1 & & & & & \ast \\
0& 1 & 0 & & & & \\
 & & 1 & 1 &  & & \\
 & & & 1 & 0  & &\\
 & & &  & \ddots & & \\
 & & & & & & 1
\end{pmatrix},
\quad
B_2 =
\begin{pmatrix}
1 & 0 & & & & & \ast \\
& 1 & 1 & & &  & \\
 & & 1 & 0 & & & \\
 & & & 1 & 1 & & \\
  & & & & \ddots & &  \\
  & & & & & & 1
\end{pmatrix}
\]

and~$(iii)$ they satisfy $\lbrack B_1, A_1\rbrack = A_1^{-2}$ and $\lbrack B_2, A_2\rbrack = A_2^{-2}$.
\end{lemm}

We conclude this subsection.
\begin{theo}\label{ex strong Massey}
The group $G$ satisfies the strong Massey Vanishing property.
\end{theo}

\begin{proof}
Let us take a family $\{\alpha_1, \dots, \alpha_n\}$ of characters in $H^1(G)$ satisfying $\alpha_i\cup \alpha_{i+1}=0$ for every $1\leq i \leq n-1$. We construct a morphism $\rho\colon G \to \mathbb{U}_{n+1}$ such that $\rho_{i,i+1}=\alpha_i$. From \cite[Proposition $2.8$]{quadrelli2023massey}, we can assume that $\alpha_i\neq 0$ for every $i$. Then from Lemma~\ref{cupzero}, we are in one of the following cases. Either:

$(a)$ for every $i$, the character $\alpha_i$ is in $H^1(G_{13})$,

$(b)$ for every $i$, the character $\alpha_i$ is in $H^1(G_{24})$,

$(c)$ for every $i$, the character $\alpha_i$ is neither in $H^1(G_{13})$ nor in $H^1(G_{24})$, but satisfies the relation:
$$\alpha_{i+1}=\alpha_i+\chi_0.$$ 

We study the case $(a)$. We observe that $G_{13}$ is in $\PP$. So~$G_{13}$ satisfies the strong Massey Vanishing property. Consequently,
there exists a morphism~$\eta_{13}\colon G_{13}\to \mathbb{U}_{n+1}$ 
such that $\eta_{13,i,i+1}=\alpha_i|_{G_{13}}$. 
Then we define:
$$\rho(x_1)=\eta_{13}(x_1), \quad \rho(x_3)=\eta_{13}(x_3), \quad \rho(x_0)=\eta_{13}(x_0), \quad \text{and } \rho(x_2)=\rho(x_4)=\mathbb{I}_{n+1}.$$

The case $(b)$ is similar, since $G_{24}$ is also in $\PP$.

For the case $(c)$, we infer that $\alpha_{2i}=\alpha_1+\chi_0$, and $\alpha_{2i+1}=\alpha_1$. We consider $A_1$, $A_2$, $B_1$ and $B_2$ the matrices defined in Lemma \ref{lemmquad}. If $\alpha_1(x_0)=1$, we take~$A\coloneq A_1$ and $B\coloneq B_1$. Else we take $A\coloneq A_2$ and $B\coloneq B_2$.
\\Let us define $\rho(x_0)\coloneq B$. Then we have $$\rho(x_0)_{2i-1,2i}=\alpha_{1}(x_0), \quad \text{and } \rho(x_0)_{2i,2i+1}=\alpha_{2}(x_0).$$
If $1\leq j \leq 4$, we observe that for every $1\leq u, v\leq n$, we have $\alpha_u(x_j)=\alpha_v(x_j)$. Thus if $\alpha_1(x_j)=0$, we define $\rho(x_j)\coloneq \mathbb{I}_{n+1}$. If $\alpha_1(x_j)=1$, we define $\rho(x_j)=A$. We obtain a morphism $\rho\colon G \to \mathbb{U}_{n+1}$ which satisifes $\rho_{i,i+1}=\alpha_i$.

\end{proof}

\subsection{Product of free groups}
Let us again consider~$\Gamma$ the square graph. The pro-$2$ group~$G_\Gamma\simeq F_{13}\times F_{24}$ is already known not to be a 
maximal pro-$2$ quotient of an absolute Galois group. Quadrelli~\cite[Theorem~$5.6$]{Quadrelli2014} showed that this group is not Bloch-Kato, i.e.\ there exists a closed subgroup of~$G_\Gamma$ with non quadratic cohomology.

From Remark~\ref{kernel uni product}, the group~$G_\Gamma$ satisfies the Kernel Unipotent property. Furthermore, a presentation of~$G_\Gamma$ is given by:
$$\langle x_1, x_2, x_3, x_4|\quad \lbrack x_1,x_2\rbrack=\lbrack x_2,x_3\rbrack= \lbrack x_3,x_4\rbrack=\lbrack x_1,x_4\rbrack=1\rangle.$$
This presentation is minimal. In fact, it is mild. We refer to \cite{Labute} for definitions. Furthermore the first author~\cite[Proposition~$1.7$]{hamza2023extensions} showed that~$\E(G_\Gamma)\simeq \E_\Gamma$. Consequently, the group~$G_\Gamma$ satisfies the Koszul property. A presentation of~$H^\bullet(G_\Gamma)$ is given by generators $\{\psi_1, \psi_2, \psi_3, \psi_4\}$ and relations:
$$\{\psi_i^2, \psi_1\cup \psi_3, \psi_2\cup \psi_4, \psi_u\cup \psi_v+\psi_v\cup \psi_u|\quad 1\leq i \leq 4, 1\leq u,v\leq 4\}.$$
In particular, a basis of~$H^2(G_\Gamma)$ is given by~$\{\psi_1\cup \psi_2, \psi_2\cup \psi_3, \psi_3\cup \psi_4, \psi_1\cup \psi_4\}$.

We recall that we denote by $F_{13}$ and $F_{24}$ the subgroups of~$G_\Gamma$ generated by~$\{x_1,x_3\}$ and $\{x_2,x_4\}$. Let us also recall that these groups are pro-$2$ free on two generators. Thus they satisfy the strong Massey Vanishing property. Furthermore, we can identify~$H^1(F_{13})$ with the vector space generated by~$\{\psi_1, \psi_3\}$ and~$H^1(F_{24})$ with the vector space generated by~$\{\psi_2, \psi_4\}$. We also observe that~$H^1(F_{13})\cup H^1(F_{13})=H^1(F_{24})\cup H^1(F_{24})=0$. Furthermore, if we take~$\alpha$ in~$H^1(F_{13})$, we observe that $\alpha\cup \psi_2=\alpha\cup\psi_4=0 \iff \alpha=0$.

Similarly to Lemma~\ref{cupzero}, we have the following result.
\begin{lemm}
Let us take~$\alpha$ and $\alpha'$ two nontrivial characters in~$H^1(G_\Gamma)$ such that~$\alpha\cup\alpha'=0$. We have the following alternative. Either:

$(i)$ $\alpha$ and $\alpha'$ are in $H^1(F_{13})$,

$(ii)$ or $\alpha$ and $\alpha'$ are in $H^1(F_{24})$,

$(iii)$ or $\alpha=\alpha'$.
\end{lemm}

\begin{proof}
The assertions~$(i)$ and~$(ii)$ are clear. Let us show~$(iii)$. We assume that~$\alpha$ is neither in~$H^1(F_{13})$ nor in~$H^1(F_{24})$. Thus there exists~$\beta$ in~$H^1(F_{13})$ different from zero and a couple $(a,b)\neq (0,0)$ in $\F_2^2$ such that $\alpha\coloneq \beta +a\psi_2+b\psi_3$. 

Let us write~$\alpha'\coloneq \beta'+a'\psi_2+b'\psi_4$ with~$\beta'\in H^1(F_{13})$. Then from~$\alpha\cup \alpha'=0$, we infer:
$$\alpha\cup \alpha'=\beta\cup \beta'+(a\beta'+a'\beta)\cup \psi_2+(b\beta'+b'\beta)\cup \psi_4=0.$$
Thus~$a\beta'=a'\beta$ and~$b\beta'=b'\beta$. Since $(a,b)\neq (0,0)$, we infer:
$$a=a', \quad b=b',\quad \beta=\beta', \quad \text{so}\quad \alpha=\alpha'.$$
\end{proof}

Consequently, as Theorem~\ref{ex strong Massey}, we can show that~$G_\Gamma$ checks the strong Massey Vanishing property.

\begin{prop}\label{square graph}
Let~$\Gamma$ be the square graph. Then the pro-$2$ RAAG $G_\Gamma$ is not the maximal pro-$2$ quotient of an absolute Galois group, but satisfies the Koszul, the Kernel Unipotent and the strong Massey Vanishing property.
\end{prop}

\begin{proof}
It remains to show that~$G_\Gamma$ checks the strong Massey Vanishing property. Let us consider a family~$\alpha \coloneq \{\alpha_1,\dots, \alpha_n\}$ of characters such that~$\alpha_i\cup \alpha_{i+1}=0$ for every~$1\leq i \leq n-1$. We construct a map $\rho\colon G_\Gamma \to \mathbb{U}_{n+1}$ such that~$\rho_{i,i+1}=\alpha_i$ for~$1\leq i \leq n$.

If~$\alpha_1$ is either in~$H^1(F_{13})$ or~$H^1(F_{24})$, we conclude using the fact that~$F_{13}$ and~$F_{24}$ are free so check the strong Massey Vanishing property.

Assume that~$\alpha_1$ is neither in~$H^1(F_{13})$ nor in~$H^1(F_{24})$. Then for every~$1\leq i \leq n-1$, we have~$\alpha_i=\alpha_{i+1}$. Let us define~$A\coloneq \mathbb{I}_{n+1}+\sum_{i=1}^n \delta_{i,i+1}$  where $\delta_{i,i+1}$ is the matrix which is~zero everywhere except in~$(i,i+1)$.

If $\alpha_1(x_i)=1$, we define~$\rho(x_i)\coloneq A$. Else~$\rho(x_i)\coloneq \mathbb{I}_{n+1}$. Since~$A$ and $\mathbb{I}_{n+1}$ commutes, the morphism~$\rho$ is well-defined. Thus~$G_\Gamma$ checks the strong Massey Vanishing property.
\end{proof}

\bibliography{bibactbib3}
\bibliographystyle{plain}
\end{document}